\documentclass[12pt,final]{article}

\usepackage{amsmath,amssymb,amsthm,color,enumerate}
\usepackage[dvipsnames]{xcolor}
\usepackage{fullpage}
\usepackage{mathtools}
\usepackage{amsfonts}
\usepackage{tikz-cd}
\usepackage{graphicx}
\usepackage{enumitem,kantlipsum}

\usepackage{datetime2}
\usepackage{comment}

\usepackage[normalem]{ulem} 

\usepackage{apptools}

\usepackage{tikz}
\usepackage{todonotes}
\usetikzlibrary{positioning}

\RequirePackage[colorlinks,citecolor=OliveGreen,urlcolor=OliveGreen,linkcolor=violet,backref=page]{hyperref}
\usepackage[marginal]{showlabels} 

\usepackage[bottom=0cm]{geometry}
\newtheorem{theorem}{Theorem}[section]
\newtheorem{lemma}[theorem]{Lemma}
\newtheorem{corollary}[theorem]{Corollary}
\newtheorem{proposition}[theorem]{Proposition}
\newtheorem{assumption}[theorem]{Assumption}

\theoremstyle{definition}

\newtheorem{example}[theorem]{Example}

\theoremstyle{remark}
\newtheorem{remark}[theorem]{Remark}

\numberwithin{equation}{section}

\newcommand{\eps}{\varepsilon}

\newcommand{\EXCLUDE}[1]{}

\newcommand{\be}{\begin{equation}}
\newcommand{\ee}{\end{equation}}

\newcommand{\bea}{\begin{eqnarray}}
\newcommand{\no}{\nonumber}

\newcommand{\eea}{\end{eqnarray}}
\newcommand\bmodif{\begin{modif}}
	\newcommand\emodif{\end{modif}}

\newcommand{\sE}{\mathbb {E} } 

\newcommand{\E}{\mathbb {E} }

\newcommand{\EXP}[1]{\mathbb {E}\!\left(#1\right) }
\newcommand{\cov}{\mathrm{Cov}} 

\newcommand{\Var}[1]{{\rm Var}\!\left(#1\right) }
\newcommand{\var}{{\rm Var}}

\newcommand{\1}[1]{\mathsf{1}\!\left[\,#1\,\right] }
\newcommand{\remove}[1]{}

\def\0{{\bf 0}}

\def\N{\mathbb{N}}

\def\mR{\mathbb{R}}

\def\mC{\mathbb{C}}

\def\mZ{\mathbb{Z}}

\def\mS{\mathbb{S}}

\def\X{{\cal X }}

\def\bK{{\bf K}}
\def\bkappa{{\boldsymbol \kappa}}

\newcommand{\md}{\mathrm{d}}

\newcommand{\cX}{{\mathcal X}}

\newcommand{\cM}{\mathcal{M} }
\newcommand{\cN}{{\mathcal N}}

\newcommand{\cL}{\mathcal{L}}
\newcommand{\cH}{{\mathcal H}}

\def\R{\mathbb{R}}
\def\N{\mathbb{N}}
\def\Z{\mathbb{Z}}

\def\P{\mathcal P}

\def\1{\mathbf{1}}

\def\C{{\cal C}}

\def\Li{\mbox{Li}_2}

\def\<{{\langle}}
\def\>{{\rangle}}

\newcommand{\dy}[1]{\textcolor{black}{#1}}
\newcommand{\mk}[1]{\textcolor{black}{#1}}

\usepackage{mlmodern}
\usepackage{eucal}
\usepackage{mathrsfs}
\usepackage[protrusion=true,expansion=true,auto=true,tracking=true]{microtype}

\title{Stationary random measures:  Covariance asymptotics,  variance bounds and central limit theorems}
\author{Manjunath Krishnapur\footnote{Department of Mathematics, Indian Institute of Science, Bangalore. Email: manju@iisc.ac.in} \, and \, 
D. Yogeshwaran\footnote{Theoretical Statistics and Mathematics Unit, Indian Statistical Institute, Bangalore. Email: d.yogesh@isibang.ac.in}}

\begin{document}
\maketitle


%
\begin{center} \textit{In fond memory of KRP, often a teacher, always a student.}
\end{center}
\abstract We consider covariance asymptotics for linear statistics of a general stationary random measure in terms of its truncated pair correlation measure. \dy{These asymptotics are particularly of interest in the case of hyperuniform random measures.} We give exact infinite series-expansion formulas for the covariance of smooth statistics of random measures involving higher-order integrals of the truncated correlation measures and higher-order derivatives of the test functions and also, equivalently in terms of their Fourier transforms. Exploiting this, we describe possible covariance and variance asymptotics for Sobolev and indicator statistics. In the smooth case, we show that the  order of variance asymptotics drops by even powers and give a simple example of a random measure exhibiting such a variance reduction. In the case of indicator statistics of $C^1$-smooth sets, we derive covariance asymptotics at surface-order scale with the limiting constant depending on the intersection of the boundaries of the two sets. We complement this with a lower bound for random measures with a non-trivial atomic part. Restricting to \dy{hyperuniform} simple point processes, we prove a central limit theorem for Sobolev and H\"{o}lder continuous statistics of simple point processes satisfying a certain integral identity for higher-order truncated correlation functions. \\ 



\tableofcontents

\section{Introduction}
\label{s:intro}

Variance asymptotics and central limit theorems of stationary random measures are well-studied topics  in stochastic geometry. The most  thoroughly investigated cases are random measures driven by Poisson or Gaussian processes. These have seen enormous progress in the last decade via analysis of their chaos expansions and Malliavin calculus; see the website \cite{MalliavinSteinwebpage} and the surveys \cite{Rossi2019,Wigman2023,Last2017,Peccati2016}. Beyond these, studies often restrict to specific classes such as perturbed lattices, determinantal processes, Gibbs measures, zeros of random analytic functions or purely atomic random measures generated by them; see, for example, \cite{Forrester1999,soshnikovclt,Nazarov12,NazarovSodin11,BYY2019,Fenzl2019,
Adhikari2020,Chen2021,Dinh2021,Yakir21,Cong2023,Feng2023,Levi2023,
Michelen2023}. Even when considering general point processes, a comprehensive description of variance asymptotics is lacking beyond number counts \cite{Lebowitz1983,Martin80,Nazarov12,Adhikari2020,Torquato2021}. In this article, covariance asymptotics for linear statistics are considered for general stationary random measures, while central limit theorems are restricted to  simple point processes. 

As a generic disclaimer, special cases of our results are known for  particular random measures, mainly point processes. The  derivations were stated and proved for the specific  examples at hand. An exposition of such derivations and ideas in wider generality is one of the aims of the article. For example,  covariance asymptotics are studied for general stationary random measures, without assuming  that the correlation measure has a density. More specific comparisons to the literature are given in discussions preceding or succeeding the corresponding results. 

Let us now spell out more details on the questions considered in this paper and our contributions. The terms and notation used here are explained in Section~\ref{sec:prelims}. Consider a stationary random measure $\X$ on $\R^d$ with intensity $\lambda$ and $f: \R^d \to \R$ that belongs to a `nice' class of functions, say smooth functions, Sobolev functions, H\"{o}lder continuous functions or indicator functions. For $L \geq 1$, let $\X_L(\cdot) = \X(L \, \cdot)$ denote the scaled random measure. We are interested in the large $L$ asymptotic behaviour of the linear statistic $\X_L(f) := \int f(x) \X_L(\md x)$, for appropriate classes of functions. It is straightforward to see from the Campbell formula \eqref{e:campbell} that $\sE[\X_L(f)] = L^{d} \lambda \int f$ if $f$ is integrable. The asymptotics of $\Var{\X_L(f)}$ and, more generally, of $\cov(\X_L(f),\X_L(g))$ are more interesting and depend  not only on $\cX$ but also on the smoothness of $f,g$. We can also ask if there are (functional) central limit theorems for suitably scaled versions of $\X_L(f)$. Apart from their role in central limit theorems, the variance asymptotics of random measures are also important in understanding the very interesting phenomenon of  {\em rigidity}, in which complete knowledge of a process in a subset may fully determine certain features of the process in the complement, see~\cite{Ghosh2015,Ghosh2017,Ghosh2017a,LRY2024}.

Even specializing to the case when $\X$ is a simple point process, variance or covariance asymptotics of linear statistics are not fully understood. Often, the two key components in second-order analysis of a stationary random measure are its truncated pair correlation measure $\bK$ and the intensity of its atomic part $\lambda_D$. Integrability of $\bK$ guarantees that $\Var{\X_L(f)} = O(L^d)$ for square-integrable $f$ and these conditions also often yield covariance asymptotics at the volume-order scale (i.e., $L^d$) with positivity of the limiting variance guaranteed if $\lambda_D + \int \bK \neq 0$. However, for many point processes (for example, perturbed lattices, determinantal point processes or zeros of the Gaussian entire function), the above condition fails, and  the point process or random measure is said to be {\em hyperuniform} (or {\em super-homogeneous}) \cite{Gabrielli2002,Torquato2003,Torquato2018,Coste2021}. A comprehensive description of covariance and variance asymptotics is lacking \dy{for hyperuniform random measures}. With this background, we now list the contributions of this article. 
\begin{enumerate}
\item We give exact formulas for covariance of smooth statistics as  series expansions in spatial and frequency domains, in Lemmas \ref{l:exactcov} and \ref{l:exactcovfourier} respectively. The terms of the series involve higher-order integrals\footnote{By higher order integral, we just mean integral against higher powers of $|t|$. } of the truncated correlation measure $\bK$ and higher-order derivatives of test functions $f,g$ or higher-order derivatives of its Fourier-transform $\widehat{\bK}$ and higher-order integrals of their Fourier-transforms $\hat{f}, \hat{g}$. These exact formulas require strong assumptions on the test statistics. But they are valuable for their clean formulations. Further, they can be modified easily for more general statistics, by truncating the series to  two orders which is the way it is often done in the literature; for example see \cite{Adhikari2020,Mastrilli2024,Yakir21,Nazarov12}.

\item These exact formulas lead to a full description of covariance asymptotics for Sobolev statistics; see Proposition~\ref{l:covasymstatfindiffsobolev}.  An interesting phenomenon is that the possible order of variance (and covariances) of smooth statistics skips alternate powers, i.e., if $\Var{\X_L(f)} = o(L^d)$ then $\Var{\X_L(f)} = O(L^{d-2})$. Further, if $\Var{\X_L(f)} = o(L^{d-2})$ then $\Var{\X_L(f)} = O(L^{d-4})$ and so on. We provide an example of a random measure achieving such variance decay rates. Examples of point processes having arbitrary polynomial decay of variances are given by Gabrielli, Joyce and Torquato \cite{Gabrielli2008} and  Lachièze-Rey \cite{LRY2024}. For Sobolev statistics of order $k$, such a phenomenon holds up to order $d-2k$. In special cases such as the Ginibre process ($d=2, k=1$) and zeros of Gaussian entire functions ($d = k = 2$) these were determined in \cite{Rider2007} and \cite{Forrester1999} respectively.  

\item At the other end of the spectrum, we give explicit covariance asymptotics for indicator statistics of $C^1$-smooth domains at scale $L^{d-1}$ (Theorem \ref{t:covindfn}) for hyperuniform $\cX$  and also surface-order variance lower bounds (i.e., $\Var{\X_L(1_{W})} = \Omega(L^{d-1})$ for $W\subseteq \R^d$ with positive volume) for indicator statistics of random measures with a non-trivial atomic component (Proposition \ref{prop:var_lb_stat_rand_meas}). The limiting covariance interestingly depends on the shared boundary of the two domains, and  in particular is zero if there is no shared boundary even if one domain is contained in the other! Such a phenomenon for the case of boxes was discovered by Lebowitz \cite{Lebowitz1983}. \dy{For more general domains and planar isotropic hyperuniform point processes, such a result was proven by Sodin, Wennman and Yakir \cite{sodinfluctuations2}, extending  earlier work of Buckley and Sodin \cite{Buckley2017} for zeros of G.E.F.} An early result of this flavour is the one-dimensional result of K. Wieand~\cite{Wieand98} (see also \cite[Theorem 6.1]{Diaconis2001}) for eigenvalues of Haar distributed unitary random matrices.  The variance lower bounds for indicator statistics were known mainly for stationary point processes; see \cite{Beck1987, Nazarov12}.

\item Finally, we state a central limit theorem for Sobolev and H\"older statistics of stationary point processes satisfying a certain integral identity for higher-order truncated correlations; see Theorems \ref{thm:cltsmooth} and \ref{thm:cltholder}. Specifically, if the variance of Sobolev statistics is of order $L^{d-2k}$, we give additional conditions on higher-order truncated correlations ensuring asymptotic normality under suitable rescaling. Such results are known for zeros of Gaussian analytic functions due to Nazarov and Sodin \cite{Nazarov12,NazarovSodin11}. We use a different approach,  via identities for multilinear cumulants (Lemma~\ref{l:cumuredu}) that bring out the cancellation in higher cumulants more clearly.  The integral identity implies hyperuniformity. A central limit theorem for indicator statistics of hyperuniform point processes was proven in \cite[Theorem 4.3]{Coste2021} (see also \cite{Adhikari2020}) assuming suitable moment bounds and variance growth conditions. In the case of statistics with growing variance (i.e., $2k < d$), our central limit theorem can be deduced straightforwardly from classical cumulant methods; for example see \cite{Nazarov12,Martin80}. However, our method works for more general test functions and can also yield a central limit theorem even in the case of variance decaying to zero as in \cite{NazarovSodin11}. 
\end{enumerate}

The rest of the paper is organized as follows: We introduce the required notation and notions for random measures in Section \ref{sec:prelims}. In Sections \ref{s:covsmooth} and \ref{s:covindicator}, we investigate covariance formulas and asymptotics for smooth statistics and indicator statistics of random measures respectively. Central limit theorems for statistics of simple \dy{hyperuniform} point processes whose truncated correlations functions satisfy a certain integral identity are given in Section \ref{s:clt}. We recall a proof of the cumulant formula in Appendix \ref{s:app_cumulant} and motivate the integral identity for truncated correlation functions by elaborating on the case of finite point processes in Appendix \ref{s:app_discrete}. In Appendix \ref{app_GAFcumulants}, we present our investigations on the integral identity for truncated correlations of zeros of Gaussian analytic functions.

\subsection{Preliminaries on random measures}
\label{sec:prelims}
We will briefly introduce some of the basics for point processes and random measures as well as setting up some notation.  For a more detailed reading on point processes and random measures,  we refer the reader to \cite{Baccelli2021,Last2017,DVJvol2,Kallenberg2017}. 

\paragraph{Notation:}
For $z=(z_1,\ldots ,z_d) \in \R^d$ and $ \alpha=(\alpha_1\ldots ,\alpha_d) \in \Z^d_+$ (where $\Z_+=\{0,1,2\ldots \}$),  set $z^{\alpha} = z_1^{\alpha_1}\ldots z_d^{\alpha_d}$,  $|z^{\alpha}|= |z_1|^{\alpha_1}\ldots |z_d|^{\alpha_d}$,  $\partial^{\alpha} = \partial_1^{\alpha_1}\ldots \partial_d^{\alpha_d}$, $\alpha!=\alpha_1!\ldots \alpha_d!$ and $|\alpha| = \sum_{i=1}^d \alpha_i$. To avoid potential confusion in the two different uses of the symbol $|\cdot |$, note that $|z|^m=(|z_1|+\ldots +|z_d|)^m$. For functions $f,g$, the inner product $\langle f, g\rangle$ and the norm $\|f\|_2 = \sqrt{\langle f,f \rangle}$ are taken in $L^2(\R^d)$.  We often  drop $\R^d$ from $L^1(\R^d)$ and $L^2(\R^d)$ and just use $L^1$ and $L^2$ respectively.   Also we abbreviate $\int_{\R^d} f(x) \md x$ by $\int f$ for $f \in L^1$ whenever the domain of integration is clear.  For example, $\langle f,g \rangle = \int fg$. All the subsets of $\R^d$ that we shall refer to are assumed to be Borel and all functions we shall refer to are assumed to be measurable.

Let $\cM$ be the set of locally-finite measures, $\cN$ be the set of locally-finite atomic measures and $\cN_s$ be the set of simple point measures (i.e., counting measures of locally finite sets), all on $\R^d$.  We equip all the spaces with evaluation $\sigma$-algebra i.e., the smallest $\sigma$-field such that $\mu \mapsto \mu(B)$ is a measurable mapping for all Borel subsets $B$.  

A \textit{random measure} $\cX$ is an $\cM$-valued random variable and an \textit{atomic random measure} $\cX$ is a $\cN$-valued random variable. If a random measure $\cX$ is  $\cN_s$-valued,  we refer to it as a  \textit{simple point process}. A random measure $\cX$ is said to be \textit{stationary} if its distribution is invariant under translations of $\R^d$ i.e.,  $\cX  + x  \overset{d}{=} \cX$ for all $x \in \R^d$,  where $(\cX + x)(B) := \cX(B - x)$ for Borel subsets $B$.  We denote the \textit{null measure} by $o$ i.e., $o(B) = 0$ for all Borel subsets $B$.  

The $k$-fold product measure of $\X$ is denoted $\X^k$ and defined by $\X^k(A_1 \times \ldots \times A_k) = \prod_{i=1}^k\X(A_i)$ for   $A_1,\ldots,A_k\subseteq \R^d$. We define the $k$th moment measure $\alpha_k(\cdot) := \sE[\X^k(\cdot)]$, a measure on $(\R^d)^k$ whenever the latter product measure has finite expectations on bounded sets. In particular, the intensity/mean measure $\alpha_1$ and the second moment measure $\alpha_2$ satisfy
\begin{align}
\sE\left[\int_{\R^d} f(x) \cX(\md x)\right]  &=  \int_{\R^d} f(x)\alpha_1(\md x),  \no \\
\label{e:campbell} \sE\left[\int_{(\R^d)^2} g(x,y) \cX^2(\md (x,y))\right]  & = \int_{(\R^d)^2} g(x,y) \alpha_2(\md (x,y)),
\end{align}
where $f : \R^d \to \R_+$  and $g : (\R^d)^{2} \to \R_+$. These identities can be extended to real-valued $f,g$ under appropriate integrability assumptions and we will refer to them as \textit{Campbell's formula} in our article.  However,  these are only special cases of Campbell's formula which holds for $\cX^n, n \geq 1$ \cite[Chapter 1]{Baccelli2021} or \cite[Section 9.5]{DVJvol2}. 

Further, we can decompose $\alpha_2$ into its diagonal component  $\alpha_D$ and {\em second factorial moment measure} $\alpha_2^!$ as follows. For $g : (\R^d)^{2} \to \R_+$,
\begin{equation}
\label{e:campbell1} \int_{(\R^d)^2} g(x,y) \alpha_2(\md (x,y)) = \int_{x \neq y} g(x,y) \alpha^!_2(\md (x,y)) + \int_{\R^d} g(x,x)\alpha_D(\md x).
\end{equation}
One can see that $\alpha_D$ is a null measure iff $\X_{\text{at}}$ (the atomic part of $\X$ as in \eqref{e:atomic}) is a null measure a.s.; see \cite[Proposition 9.5.III]{DVJvol2}.  Similarly, the  $k$th {\em factorial moment measure} $\alpha^!_k$ can be defined as follows: For $g: (\R^d)^{k} \to \R_+$ vanishing on the diagonal $\{x\in \R^d\ : \ x_i=x_j\mbox{ for some }i\not= j\}$, 
\begin{equation}
\label{e:kcorr}
\sE\left[\int_{(\R^d)^k} g(x_1,\ldots,x_k) \cX^k(\md (x_1,\ldots,x_k))\right]  = \int_{(\R^d)^k} g(x_1,\ldots,x_k) \alpha^!_k(\md (x_1,\ldots,x_k)).
\end{equation}
The density of $\alpha^!_k$, if it exists, is denoted by $\rho_k$,  the $k$-point {\em correlation function}. The definition of $\rho_k$ may be extended to the diagonal by continuity, if any, although this is not important for our analysis; see Remark \ref{rem:corr_diag}.
 
 If $\alpha_D$ is absolutely continuous with respect to the Lebesgue measure on $\{(x,x)\}\cong \R^d$, its density is denoted $\rho_D$.  If $\rho_1,\rho_2$ exist, the \textit{truncated (two-point) correlation function} is defined as $\rho_2^T(x,y) := \rho_2(x,y) - \rho_1(x)\rho_1(y)$ for $x \neq y$.

If $\cX$ is stationary,  then $\alpha, \alpha_D$ are translation-invariant and hence  $\rho, \rho_D$ exist and are constants that we denote $\lambda$ and $\lambda_D$ respectively. We refer to $\lambda$ as the intensity of $\cX$ and $\lambda_D$ as the diagonal intensity.  Further, $\alpha^!_2$ is invariant under simultaneous translations of the two arguments (i.e., $\alpha^!_2(\cdot + (x,x)) = \alpha^!_2(\cdot)$ for all $x \in \R^d$) and so is $\rho_2$, when it exists. Hence for $g : (\R^d)^{2} \to \R_+$ vanishing on the diagonal $\{(x,x)\}$, 
\begin{equation}
\label{e:campbell2} \sE\left[\int_{(\R^d)^2} g(x,y) \cX^2(\md (x,y))\right] = \int_{(\R^d)^2} g(x,x+z) \beta^!_2(\md z) \md x,
\end{equation}
for a {\em symmetric} measure $\beta^!_2$ on $\R^d \setminus \{0\}$, i.e., $\beta^!_2(A) = \beta^!_2(-A)$ for all Borel sets $A$.  The  {\em truncated correlation measure} is the  signed symmetric measure $\bK(\md z) := \beta^!_2(\md z) - \lambda^2\md z$. In case $\bK$ has a density with respect to the Lebesgue measure, we denote it by $\kappa$, i.e., $\bK(\md z) = \kappa(z) \md z$. Of course, $\kappa(z)=\rho_2^T(0,z)$.

Let  $\bK = \bK^+ - \bK^-$ be the  Hahn-Jordan decomposition of $\bK$ and set  $|\bK| := \bK^+ + \bK^-$. Further, $\|\bK\| := |\bK|(\R^d) = \bK^+(\R^d) + \bK^-(\R^d)$ is the {\em total variation} of the measure.  When $\bK$ has density, $|\bK|(\md z) = |\kappa(z)|\md z$ and the finiteness of the total variation of $\bK$  is the same as the integrability of $\kappa$. 

\subsection{Correlation functions and atomic decomposition}
\label{s:corr_atomic}

By atomic decomposition of measures \cite[Lemma 1.6]{Kallenberg2017},   we know that there exists a diffuse random measure $\X_{\text{diff}}$ and a purely atomic random measure $\X_{\text{at}}$ such that
\begin{equation}
\label{e:atomic}
\X = \X_{\text{diff}} + \X_{\text{at}}
\end{equation}
We now  describe how the  intensity and correlation functions behave under this decomposition.  Though this is well-known in random measure theory, we write it here for a better understanding.  Re-writing \eqref{e:atomic},   we know that there exists a diffuse random measure $\X_{\text{diff}}$,  a random variable $N \in \overline{\mathbb Z}_+ := \Z_+\cup\{+\infty\}$, and random vectors $X_k \in \R^d$,  $k\le N$ such that  
\begin{equation}
\label{e:atomic1}
\cX = \X_{\text{diff}} + \X_{\text{at}} =   \X_{\text{diff}} + \sum_{k=1}^{N} \xi_k \delta_{X_k}.
\end{equation}
Moreover, $\eta = \sum_{k=1}^{N} \delta_{X_k}$ is a simple point process and $\xi_k$'s are non-negative random variables. The above decomposition is unique up to the ordering of $(\xi_k,X_k)$, $k\le N$.   If $N = 0$ a.s.,  then we take $\X_{\text{at}} = o$ a.s. and so $\cX$ is diffuse. If $\X_{\text{diff}} = o$ a.s.,  then $\cX$ is purely atomic.  Trivially, $\rho_1$ is the sum of intensity functions of $\X_{\text{diff}}$ and $\X_{\text{at}}$. We have that $\X^2_{\text{diff}}(\{(x,x) : x \in \R^d\}) = 0$ a.s., and hence the diagonal intensity measure $\alpha_D$ of $\cX$ is the same as that of $\X_{\text{at}}$.  

Now we explicitly describe the intensity and two-point correlation functions in two special cases. It is easy to see that higher-order correlations can also be specified in these two cases. See \cite{Klatt2020,Torquato2021} for some computations of higher-order correlations and moments for certain point processes.
\paragraph{Purely atomic random measures:} Suppose that $\X_{\text{diff}} = o$ and assume that the simple point process $\eta$ has an intensity function $\tilde \alpha_1$ and a reduced second moment measure $\tilde \alpha^!_2$.   Then the moment measures of  $\X$ are given by
\begin{align}
\alpha_1(\md x) & =  \sE_{x}[\xi_{x}] \, \tilde \alpha_1(\md x), \qquad
\alpha_D(\md x) =  \sE_{x}[\xi_{x}^2] \, \tilde \alpha_1(\md x), \label{e:alpha1_puratom} \\
\alpha^!_2(\md (x_1,x_2)) & =  \sE_{x_1,x_2}[\xi_{x_1} \xi_{x_2}] \, \tilde \alpha^!_2(\md  (x_1,x_2)),  \label{e:alpha2red_puratom}
\end{align}
where $\sE_{x}, \sE_{x_1,x_2}$ are the first-order and second-order Palm expectations of the marked point process $\{(X_k,\xi_k)\}_{k \geq 1}$ with respect to $\eta$; see \cite[Section 3.2.4]{Baccelli2021}. Informally, $\sE_{x}[\xi_{x}]$ is the conditional expectation of $\xi_{x}$ (i.e., the mark at $x$) ``conditioned" on $x \in \eta$. Similarly, $\sE_{x_1,x_2}[\xi_{x_1} \xi_{x_2}]$ is the conditional expectation of $\xi_{x_1} \xi_{x_2}$ ``conditioned" on $x_1,x_2 \in \eta$.

Further,  if $\X_{\text{at}}$ is a simple point process ($\xi_k \in \{0,1\}$ for all $k$ a.s.) then \dy{$\alpha_D = \alpha_1$} is the intensity measure of $\X_{\text{at}}$.  In particular, for a stationary simple point process $\cX$,  $\lambda = \lambda_D$. The three common examples of point processes of interest to us are the Poisson process, Ginibre process and Zeros of the Gaussian entire function; see Section \ref{s:examples}.

\paragraph{Absolutely continuous random measure:} Suppose that $\X = \X_{\text{diff}}$ and $\X \ll  \md x$ a.s. i.e.,  there exists a random field $X : \R^d \to \R_+$ such that $\X(\md x) = X(x) \md x$ a.s..   Then we have that
\begin{align*}
\rho_1(x) &=  \sE[X(x)], \qquad
\rho_D(x) = 0,  \qquad
\rho_2(x,y) = \sE[X(x)X(y)].  
\end{align*}

\section{Covariance of smooth statistics}
\label{s:covsmooth}
 In this section,  we consider a stationary random measure $\cX$ whose   truncated correlation measure $\bK$ has fast tail decay.   In Section~\ref{s:exactformulas}, we  derive exact formulas for the covariance of linear statistics under rather restrictive assumptions on the test functions. In  Section~\ref{s:varianceupperbounds},  we obtain a general upper bound for the variance of linear statistics under minimal assumptions on the test functions, in particular for Sobolev functions.  Combining the exact formula with the variance bounds, in Section~\ref{s:asymptoticsofcovariance} we derive the asymptotics of the covariance for Sobolev statistics of the rescaled random measure. This is followed by Section~\ref{s:examples} where the conclusions are worked out for a few important examples of point processes and also, we give a simple example of a random measure which can have arbitrary decay of covariance, as indicated by  Proposition \ref{l:covasymstatfindiffsobolev}. \\

{\bf For all our results in this section and Section \ref{s:covindicator}, our standing assumption is that $\X$ is a stationary random measure with \dy{finite, non-zero} intensity $\lambda$ and the truncated correlation measure $\bK$ having finite total variation i.e., $\|\bK\| < \infty$.} \\

Though many stationary point processes satisfy this assumption, an interesting  example where this fails is the stationarized perturbed lattice. Nevertheless some of the results here hold for such processes too; see \cite{Yakir21,Mastrilli2026} and Section \ref{s:examples}.

\subsection{Exact formulas}
\label{s:exactformulas}
We first derive some exact formulas for the covariance of smooth statistics of a stationary random measure $\cX$. Although the hypotheses of the next two lemmas  are strong, the exact formula and its proof technique will be a template for many further computations in the article. 
\begin{lemma}
\label{l:exactcov}
Let $\cX$ be a stationary random measure on $\R^d$ and let $f,g:\R^d\to \R$. Assume that
\begin{enumerate}
\item  $\int |z|^m |\bK|(\md z) < \infty$ for all  $m\ge 0$.
\item $g$ is real analytic. Further, $f,g$ and   partial derivatives of $g$ of all orders are in $L^1\cap L^2$.
\item $\sum\limits_{\gamma\in \Z_+^d}\frac{1}{\gamma!}\|\partial^{\gamma}g\|_2\int |z^{\gamma}| \ |\bK| (\md z) < \infty$.
\end{enumerate}
Let $ I(\gamma):=\int z^{\gamma} \bK(\md z)$ for $\gamma\in \Z_+^d$.   Then
\begin{equation}
\label{eq:covarianceformulageneral}
 \mathrm{Cov}(\cX(f),\cX(g))  = \sum_{p\ge 0}Q_{2p}^{\bK}(f,g),
 \end{equation}
 where 
 \begin{align}\label{eq:formulaforQ} 
 Q_m^{\bK}(f,g)=\begin{cases}
 (\lambda_D + I(0))\ \langle f,g\rangle  &\mbox{ if }m=0, \\
 \sum\limits_{\gamma \in \Z^d_+:\ |\gamma|=m}\frac{1}{\gamma!}I(\gamma)\langle f,  \partial^{\gamma}g \rangle &\mbox{ if }m\ge 1.
 \end{cases}
\end{align}

\end{lemma}
Though not as strong as in the first assumption, some assumption on asymptotic decay of correlations will be present throughout the paper. The real analyticity assumption in the second one is a strong assumption on the test functions to enable us to write a clean formula for the covariance. This assumption will be relaxed. The third assumption is  to ensure that the resulting series is convergent. If $\int e^{C|z|} |\bK|(\md z) < \infty$  and $\|\partial^{\alpha}g\|_2\le C^{|\alpha|}$ for some $C$ and all $\alpha$, then this condition can be verified.
\begin{proof}
The assumption that  $f \in L^1 \cap L^2$ ensures that $\cX(f)$ and $\cX(f)^2$ have finite expectations due to \eqref{e:campbell} and \eqref{e:campbell1} respectively and hence $\X(f)$ has finite variance. Once $\cX(f)$ and $\cX(g)$ have finite variances, their covariance is well-defined. Now, using Campbell's formula and the symmetry of the truncated correlation measure,  we can write the covariance as
\begin{align}
&\cov(\cX(f),\cX(g)) =\int f(x)g(x) \alpha_D(\md x) + \int\!\!\!\int f(x)g(y)(\alpha_2^!-\lambda^2)(\md (x,y))  \nonumber \\
&= \left[\lambda_D + \int \bK \right]\int fg + \int \!\!\int [f(x)g(x-z) -f(x) g(x)] \bK(\md z)  \md x \label{eq:covariancefirtformulaprelim} 
\end{align}
The absolute integrability required for  reordering the integrals follows from the following bounds: 
\begin{align*}
\int\!\!\!\int  |f(x)| \, |g(x)| \, |\bK|(\md z)  \, \md x &=\left(\int |\bK|\right)\left(\int |fg|\right)\le \|\bK\| \, \|f\|_2 \, \|g\|_2, \\
\int\!\!\!\int |f(x)||g(x-z)| \, |\bK|(\md z) \, \md x &=\||\bK| * |g|\|_2\|f\|_2 \le \|\bK\| \|f\|_2\|g\|_2.
\end{align*}
The last inequality is true because $\bK$ has finite total variation and $g\in L^2$. 

In the second term in \eqref{eq:covariancefirtformulaprelim},  use the real analyticity of $g$ and expand $g(x-z)-g(x)$ in power series to obtain
\begin{align*}
 \int \!\!\!\int f(x)[g(x-z)-g(x)] \bK(\md z)  \md x \nonumber &=  \int \!\!\!\int  f(x) \times \sum_{\alpha\not=0}\frac{1}{\alpha!}\partial^{\alpha}g(x)z^{\alpha} \ \bK(\md z) \md x \nonumber \\
 & = \sum_{\alpha \in \mathbb Z_+^d\setminus\{0\} }\frac{1}{\alpha!} \langle f, \partial^{\alpha}g\rangle  I(\alpha). 
\end{align*}
To see that the sum and the two integrals can be ordered in any way (as we did), apply absolute values to the inner term and use the third assumption in the statement of the lemma. Finally, notice that the symmetry of $\bK$ implies that $I(\alpha)=0$ if $|\alpha|$ is odd to obtain \eqref{eq:covarianceformulageneral}.
%
%
\end{proof}
Let us restate the exact formula using Fourier transforms as is commonly done in the literature. It will be a more convenient form to work with later when we try to reduce the assumptions on the test functions. First, let us  fix our conventions regarding the Fourier transform.

For $f\in L^1$, the Fourier transform $\hat{f}$ is defined by $\hat{f}(t) = \int f(x)e^{i\<x,t \>}\md x$ for $t \in \R^d$, and for a signed measure $\mu$ with finite total variation (e.g., $\bK$), the same definition works, i.e., $\hat \mu(t) =  \int e^{i\<x,t \>}\mu(\md x)$. For $f,g\in L^1\cap L^2$,  the Plancherel identity $\<f,g\>=(2\pi)^{-d}\<\hat{f},\hat{g}\>$ holds.  As $L^1 \cap L^2$ is dense in $L^2$,  the Fourier transform extends continuously to $L^2$, and the Plancherel identity continues to hold.  If $f\in L^2$ has partial derivatives (in the weak sense) up to order $k$ and all of them are in $L^2$,  then $\widehat{\partial^{\alpha} f}(t)=(-it)^{\alpha}\hat{f}(t)$ for $|\alpha|\le k$. Conversely, if $(iz)^{\alpha}f(z)$ is in $L^2$, then its Fourier transform is $\partial^{\alpha}\hat{f}$. These are essentially the same fact because of the Fourier inversion formula $\hat{\hat{f}}(z)=(2\pi)^df(-z)$.

Just using these facts, one can rewrite $Q_{m}^{\bK}(f,g)$ in terms of Fourier transforms, as in \eqref{eq:fourierformulaforQ} below.  However, we rederive it from scratch, as the preliminary expressions will be useful later and also because the conditions on $\bK$ and $f,g$ will be changed slightly.
\begin{lemma}
\label{l:exactcovfourier}
Let $\cX$ be a stationary random measure and let $f,g:\R^d\to \R$. Assume that
\begin{enumerate}
\item $\widehat \bK$ is real analytic on $\R^d$.
\item $g$ is smooth. Further, $f,g$ and   partial derivatives of $g$ of all orders  are in $L^1\cap L^2$.
\item $\sum\limits_{\gamma \in \Z^d_+}  \frac{1}{\gamma!} |\partial^{\gamma}\widehat{\bK}(0)| \, \,  \|\partial^{\gamma}g\|_2< \infty.$
\end{enumerate}
Then \eqref{eq:covarianceformulageneral} holds where the summands can be rewritten as follows:
 \begin{align} \label{eq:fourierformulaforQ} 
 Q_m^{\bK}(f,g)=\begin{cases}
 (2\pi)^{-d}(\lambda_D + \widehat{\bK}(0))\ \int \hat{f}(t)\overline{\hat{g}(t)}\md t  &\mbox{ if }m=0, \\
 (2\pi)^{-d}\sum\limits_{\gamma \in \Z^d_+:\ |\gamma|=m}\frac{\partial^{\gamma}\widehat{\bK}(0)}{\gamma!}\int t^{\gamma}\hat{f}(t)\overline{\hat{g}(t)} \md t &\mbox{ if } m \ge 1. 
 \end{cases}
\end{align}
\end{lemma}
The conclusion of the lemma is exactly the same as that of Lemma~\ref{l:exactcov}. The equivalence of \eqref{eq:fourierformulaforQ} with \eqref{eq:formulaforQ} is also easy to see from the fact that  $\partial^{\gamma}\widehat{\bK}(0)=i^{|\gamma|}I(\gamma)$ and 
\begin{equation}
\label{e:q2mfourier}
\<f,\partial^{\gamma}g\>=(2\pi)^{-d}\<\hat{f},\widehat{\partial^{\gamma} g}\>=(2\pi)^{-d}(-i)^{|\gamma|}\int t^{\gamma}\hat{f}(t)\overline{\hat{g}(t)}\md t.
\end{equation}
But a subtle change is that compared to Lemma~\ref{l:exactcov}, in  Lemma~\ref{l:exactcovfourier} the assumption on $\bK$ is stronger (that the tail of $|\bK|$ decays faster than any polynomial only implies the smoothness of $\widehat{\bK}$, not real analyticity), while the assumption on $g$ is weaker (integrability of $t^{\gamma}\hat{f} \, \overline{\hat{g}}$ is ensured by the smoothness of $g$, there is no need for real-analyticity). 

\begin{proof}[Proof of Lemma~\ref{l:exactcovfourier}]
Rewriting \eqref{eq:covariancefirtformulaprelim} and using Plancherel's formula, we can derive
\begin{align}
\cov(\cX(f),\cX(g)) &= \<f,g*(\bK+\lambda_D\delta_0)\> =\frac{1}{(2\pi)^d}\<\hat{f},\hat{g}(\widehat{\bK}+\lambda_D)\> \nonumber \\
&=\frac{1}{(2\pi)^d}\int \hat{f}(t)\overline{\hat{g}(t)}(\widehat{\bK}(t)+\lambda_D) \md t. \label{eq:prelimcovformulafourier}
\end{align}
Since $g \in L^1 \cap L^2$ and $\bK$ has finite total variation, we have $g*(\bK+\lambda_D\delta_0) \in L^1 \cap L^2$ and hence we can use Plancherel's identity.  As $\bK$ is a symmetric measure, $\widehat{\bK}$ is real-valued and hence we omitted the conjugation on it.  Now, using the real analyticity of $\widehat{\bK}$,  we expand $\widehat{\bK}$ in power series about the origin  to obtain  
\begin{align*}
\cov(\cX(f),\cX(g)) &= \frac{(\widehat{\bK}(0)+\lambda_D)}{(2\pi)^d} \int \hat{f}(t)\overline{\hat{g}}(t)\md t \ + \ \frac{1}{(2\pi)^d}\sum_{\gamma\in \Z_+^d:|\gamma|\in 2\N} \frac{\partial^{\gamma}\widehat{\bK}(0)}{\gamma!}\int t^{\gamma} \hat{f}(t)\overline{\hat{g}(t)}\md t.
\end{align*}
The interchange of sum and integral is justified using the third assumption and \eqref{e:q2mfourier} and furthermore the terms with odd $|\gamma|$ dropped out as $\widehat{\bK}$ is an even function.  That the sum of the terms corresponding to $|\gamma|=2k$  is equal to $Q_{2k}^{\bK}(f,g)$ follows from \eqref{e:q2mfourier}.  
\end{proof}
The exact formula simplifies if $\cX$ satisfies additional symmetries.  As many random measures do have these additional symmetries, it is worth stating these formulas.  While  doing that, we also make equal assumptions on $f$ and $g$.
\begin{corollary}\label{cor:varianceunderinvariance} Let the setting be as in Lemma~\ref{l:exactcov} and in addition, assume that $f$ satisfies the same assumptions as $g$.   
\begin{enumerate}
\item {\bf Flip invariance:} If $\bK$ is invariant under reflections in the co-ordinate hyperplanes, then $I(\gamma)=0$ unless $\gamma\in (2\Z_+)^d$ and hence,
\begin{align}
Q_{2p}^{\bK}(f,g)=\sum_{|\gamma|=p}\frac{(-1)^p}{(2\gamma)!}I(2\gamma)\langle \partial^{\gamma} f,  \partial^{\gamma}g\rangle
\label{eq:covarianceformulaflip}
\end{align}

\item {\bf Orthogonal invariance:} If $\bK$ is invariant under orthogonal transformations, we may write $\bK(\md z)= r^{d-1}\bK(\md r)\sigma(\md \omega)$ where $z = r\omega$ with $r>0$, $\omega\in \mS^{d-1}$ and $\sigma$ denotes the surface-area measure on $\mS^{d-1}$. Here we abuse notation by using $\bK$ for the measure on the radial part as well but that will not cause any confusion. Then 
\begin{align}
Q_{2p}^{\bK}(f,g)  = \begin{cases}(\lambda_D + \sigma_dJ(d-1))\ \langle f,g\rangle & \mbox{ if }p=0, \\
\frac{(-1)^p\sigma_d J(d-1+2p)}{p! \ 2^p \ d(d+2)\ldots (d+2p-2)}\langle \Delta^{p/2}f, \Delta^{p/2}g\rangle & \mbox{ if }p\ge 1. \end{cases}   \label{eq:covarianceformularot} 
\end{align}
Here  $J(p):=\int_0^{\infty}r^{p} \,\bK(\md r)$ and $\Delta$ is the Laplacian, whose $\tfrac{p}{2}$ power is interpreted as $\Delta^{p/2}f:=\nabla(\Delta^q f)$ if $p=2q+1$. 
\end{enumerate}
\end{corollary}
$\bK$ satisfies flip (resp. orthogonal) invariance, if $\X$ satisfies  flip (resp. orthogonal) invariance in distribution.
\begin{proof} We start with \eqref{eq:covarianceformulageneral} and simplify it under the extra assumptions. \\

\noindent {\bf (1)} Invariance under reflections in the co-ordinate hyperplanes implies that $\bK$ is an even function of each co-ordinate.  Therefore $I(\gamma)=0$ unless $\gamma\in (2\Z_+)^d$. For $\gamma=2\beta$, integrate by parts to write $\<f,\partial^{\gamma}g\>$ as $(-1)^{|\beta|}\<\partial^{\beta}f,\partial^{\beta}g\>$ to get \eqref{eq:covarianceformulaflip}. \\

\noindent {\bf (2)} Transforming to polar co-ordinates, we have
\begin{align*}
I(2\gamma) &= \left(\int_{\R_+} r^{d-1+2|\gamma|} \bK(\md r) \right)\times \left( \int_{\mS^{d-1}}\omega_1^{2\gamma_1}\ldots \omega_d^{2\gamma_d} \ d\sigma(\omega) \right) \\
&= \sigma_dJ(d-1+2|\gamma|) B_d(2\gamma),
\end{align*}
with $\sigma_d$ and $J_p$ as in the statement and
\[
B_d(2\gamma):=\frac{1}{\sigma_d}\int_{\mS^{d-1}}\omega_1^{2\gamma_1}\ldots \omega_d^{2\gamma_d} d\sigma(\omega) = \frac{\Gamma(\frac{d}{2})\Gamma(\frac12+\gamma_1)\ldots \Gamma(\frac12+\gamma_d)}{\Gamma(\frac12)^d \ \Gamma(\frac{d}{2}+|\gamma|)}.
\]
The last formula comes from the Dirichlet integral and the fact that on the probability space $(\mS^{d-1},\sigma(\cdot)/\sigma_d)$, the random vector $(\omega_1^2,\ldots ,\omega_d^2)$ has Dirichlet($\frac12,\ldots ,\frac12$) distribution. After some simplifications,
\begin{align*}
\frac{1}{(2\gamma)!}I(2\gamma) &=  \frac{\sigma_d J(d-1+2|\gamma|)}{2^{|\gamma|} d(d+2)\ldots (d+2|\gamma|-2)}\times \frac{1}{\gamma!}.
\end{align*}
Now it is clear that for $Q_{0}^{\bK}(f,g)$ is as claimed and that for $p\ge 1$,
\begin{align*}
Q_{2p}^{\bK}(f,g) &= \frac{\sigma_d J(d-1+2p)}{2^pd(d+2)\ldots (d+2p-2)}\sum_{|\gamma|=p}\frac{1}{\gamma!}\langle f, \partial^{2\gamma}g\rangle \\
&=\frac{\sigma_d J(d-1+2p)}{p! \ 2^p \ d(d+2)\ldots (d+2p-2)}\langle f, \Delta^{p}g\rangle.
\end{align*}
Here $\Delta=\partial_1^2+\ldots +\partial_d^2$ is the Laplacian and to get there from the previous line, we just used the multinomial expansion of $\Delta^p$. Integrating by parts, we can write $\<f,\Delta^pg\>$ as $\langle \Delta^{p/2}f, \Delta^{p/2}g\rangle$, with the interpretation for odd $p$ as in the statement of the corollary. Thus  \eqref{eq:covarianceformularot} follows. 
\end{proof}

\subsection{Variance upper bounds}
\label{s:varianceupperbounds}
In this section, we wish to write explicit variance upper bounds for linear statistics with minimal assumptions on the smoothness of the test functions. This will be useful in transferring results from smooth linear statistics to more general linear statistics. For example, we use it in finding covariance asymptotics (Proposition~\ref{l:covasymstatfindiffsobolev}) and in proving central limit theorems (Theorem~\ref{thm:cltsmooth}), both with minimal smoothness assumptions.

We recall the Sobolev space $H^k=H^k(\R^d)$ (we follow the conventions and notations in~\cite[Chapter 2]{kesavan}), the space of all functions on $\R^d$ whose weak derivatives up to and including order $k$ exist and are in $L^2$. It is a Hilbert space with norm $\|f\|_{H^k}^2:=(2\pi)^{-d}\int (1+|z|^2)^{k}|\hat{f}(z)|^2 \md z$. This is a standard choice for the Sobolev norm, equivalent to another standard  definition as the square root of the sum of  squared $L^2$ norms of all partial derivatives up to and including order $k$. In the latter, if we take only $L^2$ norms of partial derivatives of order equal to $k$, then it is a pseudo-norm equivalent to  $|f|_{H^k}^2:=(2\pi)^{-d}\int |z|^{2k}|\hat{f}(z)|^2 \md z$. For example, when $k=1$, then $|f|_{H^1}^2\asymp \|\nabla f\|_{2}^2$ while $\|f\|_{H^1}^2=\|f\|_{2}^2+\|\nabla f \|_2^2$. It is a common convention to set $H^0=L^2$, for which of course $\|f\|_{H^0}^2=\|f\|_2^2 = |f|_{H^0}^2$.

\begin{lemma}\label{lem:varianceupperbounds} Let $\cX$ be a stationary random measure on $\R^d$. Assuming $\|\bK\| < \infty$, 
\[
\Var{\cX(f)}\le (\lambda_D + \|\widehat{\bK}\|_{\infty})\|f\|_2^2 \;\; \mbox{ for any }f\in L^1\cap L^2.
\]
If $\lambda_D + I(0)=0$ and $I(\gamma)=0$ for $|\gamma|<2k$ for some $k\ge 1$, then assuming that $\widehat{\bK}\in C^{2k}$ and that all its partial derivatives up to order $2k$ are uniformly bounded, with  $\|D^{2k}\widehat{\bK}\|_{\infty}:=\sup\limits_{|\gamma|= 2k}\|\partial^{\gamma}\widehat{\bK}\|_{\infty}$, we have
\[
\Var{\cX(f)}\le \frac{d^k}{(2k-1)!}\|D^{2k}\widehat{\bK}\|_{\infty}\, |f|_{H^{k}}^2 \;\; \mbox{ for any }f\in L^1\cap H^k.
\]
\end{lemma}
\begin{proof} From \eqref{eq:prelimcovformulafourier}, we get
\begin{align*}
\mbox{Var} (\cX(f)) &\le (\|\widehat{\bK}\|_{\infty}+\lambda_D)\frac{1}{(2\pi)^d}\int|\hat{f}|^2 \\
&=(\|\widehat{\bK}\|_{\infty}+\lambda_D)\|f\|_2^2.
\end{align*}

Now suppose $\lambda_D + I(0)=0$ and that $I(\gamma)=0$ for $|\gamma|<2k$ for some $k\ge 1$.  Then $Q_{j}^{\bK}(f,f)=0$ for $j<2k$. In \eqref{eq:prelimcovformulafourier}, use the Taylor expansion of $\widehat{\bK}$ to order $2k-1$ as
\[
\widehat{\bK}(z)=\widehat{\bK}(0)+\sum_{\gamma:|\gamma|\le 2k-1}\frac{\partial^{\gamma}\widehat{\bK}(0)}{\gamma!}z^{\gamma} + \sum_{\gamma:|\gamma|=2k}\frac{2k}{\gamma!}R_{\gamma}(z)z^{\gamma}\]
where the remainder term $R_{\gamma}(z)=\int_0^1(1-t)^{2k} \ \partial^{\gamma}\widehat{\bK}(tz) \md t$. As $\lambda_D + \widehat{\bK}(0)=0$ and $\partial^{\gamma}\widehat{\bK}(0) = 0$ (recall that $\widehat{\bK}(0)=I(0)$ and $\partial^{\gamma}\widehat{\bK}(0)$ is a constant times $I(\gamma)$ for $|\gamma|<2k$), \eqref{eq:prelimcovformulafourier} reduces to
\begin{align*}
\Var{\cX(f)} &\le \frac{2k}{(2\pi)^d}\sum_{|\gamma|=2k}\frac{1}{\gamma!}\int_{\R^d}\int_0^1(1-t)^{2k}|z^{\gamma}| \, \, |\hat{f}(z)|^2|\partial^{\gamma}\widehat{\bK}(tz)| \, \md t \, \md z.
\end{align*}
Bound $1-t$ by $1$ and $|\partial^{\gamma}\widehat{\bK}(tz)|$ by  $\|D^{2k}\widehat{\bK}\|_{\infty}$.  When $|z^{\gamma}|/\gamma!$ is summed over $|\gamma|=2k$, we get $(|z_1|+\ldots +|z_d|)^{2k}/(2k)!$ (multinomial expansion), which can be bounded by $d^k|z|^{2k}/(2k)!$ (where $|z|^2=|z_1|^2+\ldots +|z_d|^2$).  Putting these together, we  arrive at
\begin{align*}
\Var{\cX(f)} &\le \frac{d^k}{(2k-1)!}\|D^{2k}\widehat{\bK}\|_{\infty} \, |f|_{H^k}^2,
\end{align*}
which completes the proof.
\end{proof}


\subsection{Covariance asymptotics}\label{s:asymptoticsofcovariance}
For $L\ge 1$, consider the scaled random measures $\cX_L$, defined as $\cX_L(A)=\cX(LA)$ for  sets $A\subseteq \R^d$ (we could of course define $\cX_L$ for $L>0$, but we are interested in the asymptotics as $L\to \infty$, and the uniformity of certain estimates breaks down when $L\to 0$). What to expect is made clear by Lemma~\ref{l:exactcov} or  Lemma~\ref{l:exactcovfourier} and we will now formalize this. 

Let us denote the quantities associated to $\cX_L$ by adding the subscript $L$. From \eqref{e:campbell2}, we derive that 
 $\bK_L(A) = L^d\bK(LA)$ for any Borel $A \subset \R^d$, and hence by a straightforward change of variable,  we obtain that $I_L(\gamma) =  \int z^{\gamma} \bK_L(\md z) = L^{d-|\gamma|}I(\gamma)$. Thus $Q_m^{\bK_L}(f,g)=L^{d-m}Q_m^{\bK}(f,g)$.  For $f,g$ satisfying the assumptions of Lemma~\ref{l:exactcov}, we therefore get
\begin{equation}
\label{eq:Cov_expansion1} 
\cov(\cX_L(f),\cX_L(g))= \sum_{k=0}^{\infty}L^{d-2k}Q_{2k}^{\bK}(f,g).
\end{equation}
As $L\to \infty$, the first non-zero summand dominates and that determines the order of the covariance as well as the limiting covariance (of that order). Observe the surprising feature that if the covariance is $o(L^d)$, then it must be $O(L^{d-2})$ and so on, getting only alternate powers of $L$.   Later in Section \ref{s:examples},  we shall see an example of random measure showing that all these orders are achievable even if $ \langle f,  \partial^{\gamma}g\rangle$ are non-vanishing for all $\gamma \in \Z^d_+$.  

As it stands, the above proof only works for $f,g$ satisfying the hypotheses of Lemma~\ref{l:exactcov}, and those are too restrictive. For example, if $\lambda_D+I(0)=0$ and $I(\gamma)=0$ for $|\gamma|<2k$ for some $k\ge 1$, then the above proof approach suggests that $\cov(\cX(f),\cX(g))\sim L^{d-2k}Q_{2k}^{\bK}(f,g)$.  Considering that  $Q_{2k}^{\bK}(f,g)$ is well-defined if  $f,g\in H^{k}$, one may expect that this should be the only required condition in addition to $f,g\in L^1$.  Indeed, our variance bound in Lemma~\ref{lem:varianceupperbounds} was proved precisely for functions in $H^k\cap L^1$. 

\dy{Random measures satisfying $\lambda_D + I(0) = \lambda_D + \hat{\bK}(0) = 0$ are said to be {\em hyperuniform} and in such cases, variance of linear statistics are $o(L^d)$. There is now an extensive body of work on hyperuniform systems; see the surveys \cite{Torquato2018,Coste2021,LRY2025}. The next proposition shows some curious restrictions on the possible variance asymptotics for hyperuniform random measures.}
\begin{proposition}
\label{l:covasymstatfindiffsobolev}
Let $\cX$ be a stationary random measure and let $\cX_L$ denote the scaled random measure defined by $\cX_L(A)=\cX(LA)$ for $L \geq 1$. Assume that $\int |z^{\gamma}||\bK|(\md z) < \infty$ for $|\gamma|\le 2k+1$. Let $f,g\in L^1\cap H^k$ functions on $\R^d$ for some $k\ge 0$ (with the convention that $H^0=L^2$).  Then,
\begin{enumerate}
\item $L^{-d} \cov(\cX_L(f),\cX_L(g))  \to Q_0^{\bK}(f,g)$ as $L\to \infty$.
\item If $\lambda_D + I(0) = 0$ and $I(\alpha)=0$ for  $1\le |\alpha|< 2k$, then
$L^{-d+2k} \cov(\cX_L(f),\cX_L(g)) \to Q_{2k}^{\bK}(f,g)$.
\item If $\lambda_D + I(0)= 0$ and $I(\alpha)=0$ for  $1\le |\alpha|\le 2k$, then $\cov(\cX_L(f),\cX_L(g))=O(L^{d-2k-1})$  if one of  $f$ or $g$ is in $C_b^{2k+1}$  (functions with bounded, continuous partial derivatives of order $2k+1$).
\end{enumerate}
\end{proposition}
We prove the proposition in two steps.  First we give a proof that works when at least one of $f$ or $g$ is in $C_b^{2k+1}$. Next, we use the fact that $C_c^{\infty}$ is dense in $H^k$, and the variance bound from Lemma~\ref{lem:varianceupperbounds} to extend the result to $f,g\in L^1\cap H^k$. This approximation strategy is adapted from Rider and Vir\'{a}g~\cite{Rider2007} who did a similar thing for the Ginibre ensemble with $k=1$.

\begin{proof}[Proof of Proposition~\ref{l:covasymstatfindiffsobolev}] \

\noindent \textsc{\underline{STEP 1:} Proof when one of $f,g$ is in $C_b^{2k+1}$.} \\

Assume that $g\in C_b^{2k+1}$ and consider the Taylor expansion
\[
g(x-z)-g(x) = \sum_{\alpha\in \Z_+^d: 1\le |\alpha|\le 2k}\frac{(-1)^{|\alpha|}}{\alpha!}\partial^{\alpha}g(x) z^{\alpha} \ + \ \sum_{\alpha\in \Z_+^d:  |\alpha|=2k+1}\frac{(-1)^{|\alpha|}|\alpha|}{\alpha!}R_{\alpha}^g(x,z)z^{\alpha}
\] 
with remainder term $R_{\alpha}^g(x,z)=\int_0^1(1-t)^{2k+1}\partial^{\alpha}g(x-tz) \md t$. From this it is clear that $|R_{\alpha}^g(x,z)|\le \|\partial^{\alpha}g\|_{\infty}$ and that $R_{\alpha}^g$ is continuous in $x,z$ (hence the integrals below are legitimate). Apply formula \eqref{eq:covariancefirtformulaprelim} to $\cX_L$ and observe that $\bK_L(\md z) = L^{2d}\bK(\md(Lz))$ to get  (recall that $I(\alpha)=0$ for odd $|\alpha|$)
\begin{align}
 \cov(\cX_L(f),\cX_L(g)) =& L^d(\lambda_D + I(0))\<f,g\>+\sum_{1\le |\alpha|\le 2k, |\alpha|\in 2\N}L^{d-|\alpha|}\frac{1}{\alpha!}I(\alpha)\<f,\partial^{\alpha}g\> \nonumber \\
 &-L^{d-2k-1}\sum_{|\alpha|=2k+1}\frac{(2k+1)}{\alpha!}\int \!\! \int f(x) R_{\alpha}^g\left(x,\frac{z}{L}\right)z^{\alpha} \bK(\md z) \md x. \label{eq:covwithremainderterm}
\end{align}
For $\alpha\in \Z_+^d$ with $|\alpha|=2k+1$,
\begin{align*}
\big| \int \!\! \int f(x) R_{\alpha}^g\left(x,\frac{z}{L}\right)z^{\alpha}  \bK(\md z) \md x \big|&\le \|\partial^{\alpha}g\|_{\infty} \|f\|_{L^1} \int|z^{\alpha}| |\bK|(\md z).
\end{align*}
This bound is independent of $L$, showing that the third term on the right in \eqref{eq:covwithremainderterm} is $O(L^{d-2k-1})$. The coefficient of $L^{d-2j}$ for $0\le j\le k$ is precisely $Q_{2j}^{\bK}(f,g)$, by \eqref{eq:formulaforQ}. Thus,
\[
\cov(\cX_L(f),\cX_L(g)) = \sum_{j=0}^kL^{d-2j}Q_{2j}^{\bK}(f,g) \; + \; O(L^{d-2k-1}).
\]
All three statements of the proposition follow immediately when one of $f$ or $g$ is in $C_b^{2k+1}$. It remains to prove the second statement for more general $f,g$. \\

\noindent \textsc{\underline{STEP 2:} Proof for $f,g\in L^1\cap H^k$.} \\

It is well-known (Theorem~2.1.3 in \cite{kesavan}) that $C_c^{\infty}$ (and hence $C_b^{2k+1}$) is dense in $H^k$. Given $f,g$, there exist $f_n,g_n\in C_c^{\infty}$ such that $\|f_n-f\|_{H^k}\to 0$ and $\|g_n-g\|_{H^k}\to 0$. 

Consider the first two cases: that either $I(0)\not=-\lambda_D$ or $I(\gamma)\not=0$ for some $|\gamma| \le  2k$. We have proved that for fixed $n$, as $L\to \infty$,
\[
L^{-d+2k}\cov(\cX_L(f_n),\cX_L(g_n))\to Q_{2k}^{\bK}(f_n,g_n).
\]
 Further, $Q_{2k}^{\bK}(f_n,g_n)\to Q_{2k}^{\bK}(f,g)$ as $n\to \infty$ as seen from  \eqref{eq:formulaforQ}. Now,
\begin{align*}
\;\; &|\cov(\cX_L(f),\cX_L(g))-\cov(\cX_L(f_n),\cX_L(g_n))|\\
\;\; &\le |\cov(\cX_L(f),\cX_L(g-g_n))|+|\cov(\cX_L(f-f_n),\cX_L(g_n))| \\
\;\; &\le \sqrt{\Var{\cX_L(f)}}\sqrt{\Var{\cX_L(g-g_n)}}+\sqrt{\Var{\cX_L(g_n)}}\sqrt{\Var{\cX_L(f-f_n)}} \\
&\le CL^{d-2k}\{|f|_{H^k}|g-g_n|_{H^k}+|g_n|_{H^k}|f-f_n|_{H^k}\}.
\end{align*}
In the last line, we invoked Lemma~\ref{lem:varianceupperbounds}, using the observation that  $\widehat{\bK}_L(t)=L^d\widehat{\bK}(t/L)$ and consequently $\|D^{2k}\widehat{\bK}_L\|_{\infty}\le L^{d-2k}\|D^{2k}\widehat{\bK}\|_{\infty}$. Here we also use that the integrability condition on $\bK$ implies the smoothness of $\widehat{\bK}$ as required in Lemma~\ref{lem:varianceupperbounds}. Clearly $|f-f_n|_{H^k}\le \|f-f_n\|_{H^k}\to 0$ as $n\to \infty$, and similarly $|g_n-g|_{H^k} \to 0$. In particular, $|g_n|_{H^k}\to |g|_{H^k}$. Thus, letting $L\to \infty$ and then $n\to \infty$,  we see that $L^{-d+2k}\cov(\cX_L(f),\cX_L(g))\to Q_{2k}^{\bK}(f,g)$.  
\end{proof}
The proof in STEP 1 above can be regarded as a modification of the proof of Lemma~\ref{l:exactcov}. Instead we could have modified the proof of Lemma~\ref{l:exactcovfourier} using the  Taylor expansion of $\widehat{\bK}$ to order $2k$ (similar to what we did in the proof of Lemma~\ref{lem:varianceupperbounds}) to get an almost identical result. Covariance asymptotics as above for analytic functions has been obtained recently in \cite[Theorem 5.1]{Mastrilli2024} by more carefully analyzing \eqref{eq:prelimcovformulafourier}.

\subsection{Examples}
\label{s:examples}
We investigate the asymptotics of $\cov(\cX_L(f),\cX_L(g))$ explicitly for some specific random measures and simple point processes. Though the variance asymptotics for the simple point processes are derived elsewhere by the specific structure of the point process under consideration, here we show how they follow from our general result (Proposition \ref{l:covasymstatfindiffsobolev}) and also we obtain covariance asymptotics without additional work.  We also provide a simple example of a random measure that has arbitrary polynomial decay of covariance but it is more challenging to find an example of a point process with arbitrary decay of covariance as it is not easy to guarantee existence of stationary point processes with a specified truncated correlation function; see \cite{Ambartzumian1991,Caglioti2006}. Nevertheless, using explicit analysis of structure factor (i.e., $\lambda_D + \widehat{\bK}(\cdot)$), point processes with arbitrary polynomial decay of structure factor (or equivalently covariance asymptotics) has been indicated in \cite[Section III.B]{Gabrielli2008}. 

We will not define the point processes formally but refer to \cite{Ben09,Last2017,Baccelli2021,Coste2021} for definitions and also more examples of point processes. We also describe higher-order truncated correlations for determinantal processes and zeros of Gaussian entire function in Section \ref{s:highertruncorr}. \dy{Though the main focus of the article is on hyperuniform random measures, we begin with the well-known example of Poisson process for completeness and to contrast with the other examples.}
\begin{enumerate}
\item {\em The stationary Poisson process.} See \cite{Last2017} for definition.  The correlations are given by\\ $\rho_k(x_1,\ldots ,x_k)=\lambda^k$ for some $\lambda>0$.  Indeed this characterizes the stationary Poisson point process. In this case,  $\kappa(z)=0$ and hence $I(0)=0$. Thus,  $\cov(\cX_L(f),\cX_L(g))\sim \lambda L^d\langle f,g\rangle$.  That is,  the variance is volume order and in the limit we get white noise.  

\item {\em The Ginibre example.}  See \cite[Section 6.4]{Ben09} for definition.  We just recall that this is a point process on $\R^2 \cong \mC$ whose correlations are given by $\rho_k(z_1,\ldots ,z_k)=\det(K(z_i,z_j))_{1\le i,j\le k}$ where $K(z,w)=\frac{1}{\pi}e^{-\frac12|z|^2-\frac12|w|^2+z\overline{w}}$.  From the first two correlations we see that $\lambda=1/\pi$ and $\kappa(z)=-\frac{1}{\pi^2}e^{- |z|^2}$ and hence $I(0)=-\lambda$ but \dy{$I(\alpha) \neq 0$ for $\alpha = (2,0)$ or $\alpha = (0,2)$}.  In fact, (since the process is isotropic) using Proposition~\ref{l:covasymstatfindiffsobolev} and Corollary \ref{cor:varianceunderinvariance},  we have that for $f, g \in L^1 \cap H^1$,  
\[ \lim_{L \to \infty} \cov(\cX_L(f),\cX_L(g)) =   -  \frac{\pi J(3)}{2}\langle \nabla f , \nabla g\rangle = \frac{1}{4\pi} \langle \nabla f , \nabla g\rangle 
\]
where we have used
\[
\sigma_2=2\pi,   \quad J(3) = -\frac{1}{\pi^2}\int_{\mathbb R} r^3e^{-r^2}dr=-\frac{1}{2\pi^2}.
\]
This was first derived in Rider and Vir\'{a}g~\cite[Lemmas 12 and 14]{Rider2007}. 

\item {\em Zeros of the G.E.F. }  See \cite[Chapter 2]{Ben09} for more details on this point process.  This is another point process on $\R^2 \cong \mC$.  Let $F(z)=\sum_{n\ge 0}a_n\frac{z^n}{\sqrt{n!}}$ where $a_n$ are i.i.d. standard complex Gaussians, then $F$ is a Gaussian analytic function on the whole plane with covariance $\E[F(z)\overline{F(w)}]=e^{z\overline{w}}$. The zeros of $F$ form a stationary point process. Well-known Kac-Rice formulas allow to express the correlation functions of the zero set in terms of the covariance kernel. In particular, it has intensity $\frac{1}{\pi}$ and $\rho_2(z,w)-\rho_1(z)\rho_1(w)=\frac{1}{4\pi^2}\Delta_z\Delta_w \sum_{m\ge 1}\frac{1}{4m^2}e^{-m|z-w|^2}$. After some computation, 
\begin{align}
\kappa(z)
&=-\frac{1}{\pi^2} \;  \sum_{m\ge 1}e^{- m|z|^2} (2-4m|z|^2+m^2|z|^4). \label{eq:kappagaf}
\end{align}
Integration by parts (transfer the Laplacians from $\kappa$ to the test functions), it is clear that $I(2\gamma)=0$ for $|\gamma|\le 1$. Hence using Proposition~\ref{l:covasymstatfindiffsobolev} and Corollary \ref{cor:varianceunderinvariance} with $p=2$,  we have that for $f, g \in L^1 \cap H^2$,  
\begin{align*} \lim_{L \to \infty} L^2\cov(\cX_L(f),\cX_L(g)) &= \frac{2\pi J(5)}{2! \times 2^2 \times 2(2+2)}\times \langle \Delta f,\Delta g\rangle   \\
&=\frac{\zeta(3)}{16\pi} \langle \Delta f,\Delta g\rangle.
\end{align*}
In the last line, we used \eqref{eq:kappagaf} to compute $J(5)$ explicitly. This expression was derived by Forrester and Honner~\cite{Forrester1999}. Also see \cite{sodintsirelson1}  and \cite{NazarovSodin11} for more on variance asymptotics for various statistics (including indicator statistics) of zeros of G.E.F.. 

Later in Example \ref{ex:GEFzerostrunccorr}, we shall derive a formula for truncated correlation functions for general Gaussian analytic functions.

\item {\em Random measures with fast decay of covariance.} For any $p \geq 1$, we construct a random measure $\X$ that is absolutely continuous w.r.t. to Lebesgue measure as in Section \ref{s:corr_atomic} such that $I(\alpha) = 0$ for all $0 \leq |\alpha| < 2p$ i.e., $L^{-d+2k}\Var{\X_L(f)} \to 0$ as $L \to \infty$ for $f$ smooth and $k < p$. In other words, there exist random measures with arbitrary fast decay of covariance.  

Let $\varphi : \R^d \to \R$ be such that $\int \varphi = 0$, $\int |\varphi| = 1$ and $|\hat{\varphi}(x)| = o(\|x\|^p)$ as $x \to 0$ for some $p \geq 1$. For example, we can take $\varphi$ to be the $p$-fold convolution of $x \mapsto \prod_{i=1}^d \left( 1_{[0,\frac{1}{2}]}(x_i) - 1_{[\frac{1}{2},1]}(x_i) \right)$. Let $Y(x), x \in \R^d$ be a symmetric stationary stochastic process with a finite range of dependence and $|Y(x)| \leq 1$. Set $\sigma_Y(x) = \cov(Y(x),Y(0)) = \E[Y(0)Y(x)], x \in \R^d$ to be the {\em covariance kernel} of $Y$.  Define the random measure
$$ \X (\cdot) = \int_{\cdot} X(x) \, \md x \, \, \text{where} \, \, X(x) = 1 + (Y * \varphi)(x) = 1 + \int Y(z) \, \varphi(x-z) \, \md z.$$
 Setting $Z := Y * \varphi$, observing that $X = 1 + Z$, $\sigma_X(z) = \sigma_Z(z) := \cov(Z(0),Z(z)) = \E[Z(0)Z(z)]$ as it can be verified that $\E[Z(\cdot)] \equiv 0$. Since $|Z| \leq 1$, $X \geq 0$
 
Clearly $\X$ is a stationary random measure as $X$ is a non-negative stationary random field. Thus $\rho_1 \equiv \E[X(0)] = 1$, $\rho_2(x,y) = \E[X(x)X(y)] =: \sigma_X(x-y) + 1$ and the truncated correlation density is $\kappa = \sigma_X$.  We can now compute that $\hat{\sigma}_Z(z) = \hat{\sigma}_Y(z)|\hat{\varphi}(z)|^2$.  From the above observations, we have that $\X$ is a stationary random measure such that $\hat{\kappa}(z) = \hat{\sigma}_Y(z)|\hat{\varphi}(z)|^2$. Our assumption on decay of $\hat{\varphi}$ yields that $|\hat{\kappa}(z)| = o(\|z\|^{2p})$ as $x \to 0$ and so from Lemma \ref{l:exactcovfourier}, we have that $I(\alpha) = 0$ for  all $0 \leq |\alpha| < 2p$ as required.

\item {\em Perturbed lattices.} Though this example does not satisfy the common assumption on integrable truncated two-point correlation measure, still many results can be proven; see \cite{Gacs1975,Gabrielli2004,Peres2014,Yakir21,Mastrilli2026} for example. Let $\xi_z, z \in \mZ^d$ be i.i.d. random vectors in $\mR^d$ and $U$ be a uniform random vector in $[0,1]^d$. Then we define $\cX = \{z + \xi_z : z \in \mZ^d\}$ and $\cX_s = \{z + \xi_z + U : z \in \mZ^d\}$ as the perturbed lattice and stationarized perturbed lattice respectively. Though $\cX_s$ is stationary (but $\cX$ is not), it can be shown that it does not have an integrable truncated two-point correlation measure. Given the structure of the point process, explicit formulas are known for $\widehat{\bK}$; see \cite{Gabrielli2004} and from which, one can conclude the variance of smooth statistics of $\cX$  (or $\cX_s$) are at least of order $L^{d-2}$. In the case, when $\xi_z$'s are Gaussian this has been proven for $\cX$ but it fails for $\cX_s$; see \cite[Theorem 4]{Yakir21}. 

\item {\em Point process with  fast decay of variance - Multiple perturbations of a lattice.} As mentioned above perturbed lattices cannot achieve variance growth of order less than $L^{d-2}$. However, \cite{Gabrielli2008} and \cite{LRY2024} provided examples achieving smaller variance growth rates by using multiple perturbations of each lattice point that preserve the center of mass and higher moments. Specific constructions achieving variance growth of order  $L^{d-4}$ can be found in  \cite{Gabrielli2008} and \cite[Pg 5]{sodintsirelson1} (the example of third toy model) with  the former also indicating examples achieving variance growth of orders $L^{d-6}, L^{d-8}$ and so on. 

\item {\em Signed random measures.} It is also of interest to consider signed random measures and it arises naturally in many examples; for example, see \cite{BYY2019}.  We remark that it is possible to extend some of our results to many natural examples of signed random measures provided they can be written as difference of two (non-negative) random measures and applying our results individually to the two random measures. One needs to be careful as $\int f(x) \, \cX(\md x)$ is undefined in general for a signed random measure $\cX$ on $\mR^d$ when both the positive and negative components have infinite total measure. 
\end{enumerate}

\bigskip

\section{Covariance of indicator statistics}
\label{s:covindicator}

In the previous section,  we derived exact formulas as well as asymptotics for $\cov(\cX(f), \cX(g))$ when $f,g$ satisfy certain smoothness or analyticity conditions.   The question remains as to what is the behaviour when $f,g$ are less regular. In Section \ref{s:covasymptotindicator},  we first consider covariance asymptotics for the case when $f,g$ are indicator functions of bounded open sets with $C^1$-smooth boundaries.  Then in Section \ref{s:varlb},  we show surface-order variance lower bound for $\Var{\cX_L(W)}$ for sets $W$ with positive Lebesgue measure and when $\cX$ has a non-trivial atomic part i.e., $\lambda_D > 0$.  

\subsection{Covariance asymptotics}
\label{s:covasymptotindicator}
We begin by stating the main theorem of this section.
\begin{theorem}
\label{t:covindfn}
Let $\cX$ be a stationary random measure whose truncated correlation measure $\bK$ has finite total variation. Assume that $A$ and $B$ are bounded open sets with $C^1$-smooth boundaries.  Let $\cX_L$ be the scaled random measure defined by $\cX_L(A) := \cX(LA)$,  $L \geq 1$.  Then
\begin{enumerate}
\item $L^{-d}\cov(\cX_L(A),\cX_L(B)) \to Q_0^{\bK}(f,g) = (\lambda_D + \int \bK ) \, |A \cap B|, $ where $| \cdot |$ is the Lebesgue measure. 
\item If $\lambda_D + \int \bK = 0$ and further $\int |z||\bK|(\md z)<\infty$, then it holds that
\begin{equation}
\label{eq:covasymp_surface}
L^{-d+1}\cov(\cX_L(A),\cX_L(B))  \to  -\frac12 \int_{\partial A\cap \partial B} \eps(x) \left[\int |z\cdot N_A(x)| \ \bK(\md z) \, \right] \, \sigma_A(\md x),
\end{equation}
where $N_A(x)$ is the exterior normal to $A$ at $x$ and $\sigma_A$ is the surface area measure of $\partial A$. Further, $\eps(x) \in \{+1,0,-1\}$ is defined as follows: $\eps(x) = +1$ (resp. $-1$) if $N_A(x)$ is the exterior (resp. interior) normal to $B$ at $x$, and $0$ otherwise. 
\end{enumerate}
\end{theorem}
As the proof will show, $\eps = \pm 1$ a.e. on $\partial A \cap \partial B$ and hence the role of $A$ and $B$ can be interchanged in the RHS of \eqref{eq:covasymp_surface}. The first statement in the above theorem is same as that of Proposition~\ref{l:covasymstatfindiffsobolev}(i) for $f,g \in H^0$ but without requiring any further differentiability.   This can be derived easily from \eqref{eq:covariancefirtformulaprelim}.  But the second statement is non-trivial and cannot be deduced from the results of the previous section. Recall that precursor results of this nature for point processes due to Lebowitz, Wieand, Buckley and Sodin were discussed in the introduction. A special case of the second statement is when $A =B$ i.e., variance asymptotics. Note that even in the case of variance, the positivity of the limit is not claimed in the statement.  Our upcoming Proposition \ref{prop:var_lb_stat_rand_meas} will guarantee this to be the right order of asymptotics in the case of stationary random measures with a non-trivial atomic part. We refer the reader to discussion below Proposition \ref{prop:var_lb_stat_rand_meas} and in particular, \eqref{e:varasymp_surface}. We expect that the above theorem should hold for more general domains, say Lipschitz domains \dy{or even sets of finite perimeter}. \dy{For example, covariance asymptotics for planar isotropic hyperuniform point proceses have been shown for more general domains in \cite[Theorem 1.3]{sodinfluctuations2}. Precursors of such results for specific point processes or under restrictions on the domains can be found in  \cite[Section III]{Lebowitz1983}, \cite{Wieand98,Diaconis2001} and \cite{Buckley2017}. Variance asymptotics for hyperuniform point processes for sets of finite perimeter or under weaker integrability conditions has been shown recently; see \cite{Levi2023} and \cite[Theorem 1.5]{Antezana2026}. The latter pre-print appeared during the revision of this paper.}

Again as in \eqref{eq:covarianceformularot}, the limiting covariance formula has a neat product form when $\bK$ is isotropic. Assuming that $\bK$ is isotropic and abusing notation, we write $\bK(\md z) = r^{d-1}\bK(\md  r) \sigma(\md w)$ with $z = rw$ for $r > 0$ and $w \in \mS^{d-1}$. Then, the conclusion of Theorem \ref{t:covindfn}(2) can be written as
\begin{align*}
 L^{-d+1}\cov(\cX_L(A),\cX_L(B))  & \to  -\frac12 \left[ \int_{\partial A\cap \partial B} \eps(x)   \int_{\mS^{d-1}} |w \cdot N_A(x)| \, \sigma(\md w) \, \sigma_A(\md x) \right] \times  \int_0^{\infty} r^d \, \bK(\md r) \\
& = c_d \, \int_{\partial A\cap \partial B} \eps(x) \sigma_A(\md x)  \times \int_0^{\infty} r^d \, \bK(\md r) , 
 \end{align*}
where 
$$ c_d := \frac{1}{2} \, \int_{\mS^{d-1}} |w_1| \, \sigma(\md w) = \frac{\sigma(\mS^{d-1})\Gamma(\frac{d}{2})}{2\sqrt{\pi}\Gamma(\frac{d+1}{2})}.$$ 
%
%
%
Observe that the limiting covariance is zero if $A$ and $B$ do not share a boundary, even if one of them is contained in the other. On the other hand, if they share a part of the boundary and their interiors are disjoint, then the covariance is negative and proportional to the length of the shared boundary. But if one is contained in the other and they share a part of the boundary, then the covariance is positive. 
\begin{figure}
\label{fig:delAB}
  \centering  \includegraphics[scale=0.6]{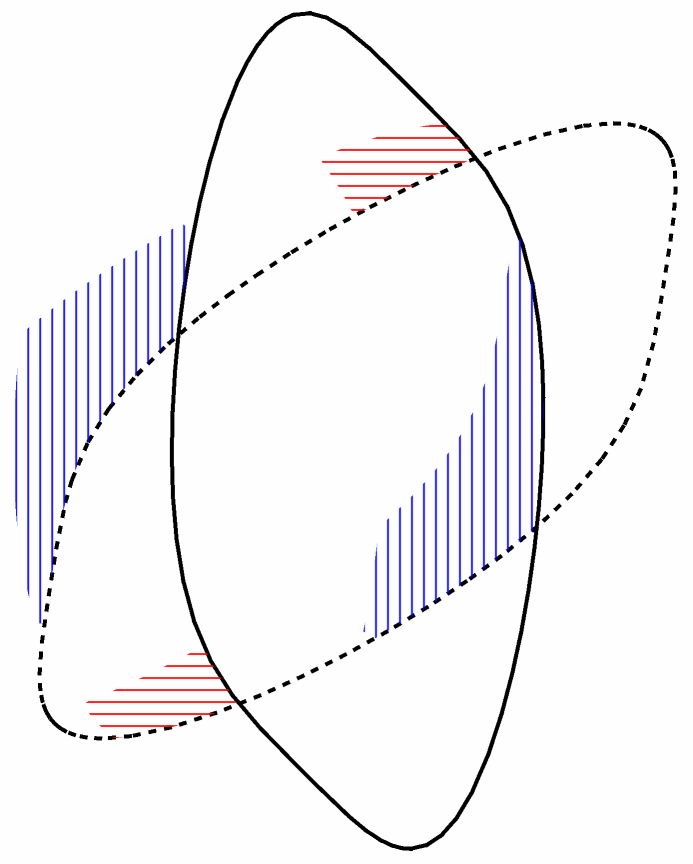}
\mk{    \caption{$A$ and $B$ are the regions bounded by the solid and dashed contours respectively. For a small $s$, the integrand of  $h_s(z)$ with $z=(1,0)$ takes value $+1$ on the vertically (blue) shaded regions and $-1$ on the horizontally (red) shaded regions.}}
\end{figure}
\dy{
The proof relies on the geometric lemma below. To state it, for nice enough sets $A,B\subseteq \R^d$, define the functions $h_s:\R^d\to \R$ for $s\ge 0$ by
\begin{align}
h_{s}(z)&:= \frac{1}{s}\int_{\R^d} (1_{(A-sz)\setminus A}-1_{A\setminus (A-sz)})(1_{(B-sz)\setminus B}-1_{B\setminus (B-sz)})  \mbox{ if }s>0, \label{defn:hs}\\
h_0(z)&:= \int_{\partial A\cap \partial B}\eps(x) |z\cdot N_A(x)|  \ \sigma_A(\md x) = \int_{\partial A\cap \partial B}\eps(x) |z\cdot N_B(x)|  \ \sigma_B(\md x). \label{defn:h0}
\end{align}
The function $h_0$ is almost in the form that occurs in  Theorem~\ref{t:covindfn}. It measures the signed common boundary in the direction $z$, as it vanishes unless $\partial A$ and $\partial B$ intersect and  have a common normal on part of the intersection. To get  $h_s(z)$, displace $A$ and $B$ in the direction $z$ by a distance $s$ (to be thought of as small), and  measure the signed area common to the regions swept out (see Figure \ref{fig:delAB} for illustration). This is easiest to see for  rectangles $A=[a,b]\times [c,d]$, $B=[u,v]\times [w,x]$ and $z=(1,0)$. Then if $u=a$ but $v\not=b$, we see that  $h_s(z)$ is equal to  the length of $[c,d]\cap [w,x]$ for small enough $s$. On the other hand, if $u=b<v$, then  we get the same, but with a negative sign. In all cases where $A$ and $B$ do not share a part of their vertical edges, we have that $h_s(z)= 0$ for small enough $s$. Thus, $h_s = h_0$ for small enough $s$. However, for general $z$ (and general `smooth' $A,B$), $h_s$ is an approximation of $h_0$, which is what the geometric lemma below says. If one proves it for more general sets, the covariance result Theorem~\ref{t:covindfn} can be correspondingly generalized.
\begin{lemma}\label{lem:geometricboundary} Let $A$ and $B$ be bounded open sets in $\R^d$ with $C^1$ boundary. Define $h_s$ and $h_0$ by \eqref{defn:hs} and \eqref{defn:h0} respectively. Then $h_s(z)\to h_0(z)$ for all $z\in \R^d$, as $s\downarrow 0$.
\end{lemma}
We first prove the main theorem, assuming this lemma.
}
\begin{proof}[Proof of Theorem~\ref{t:covindfn}]
As explained above the first statement is trivial and hence we prove the second alone. Let $f=1_A$ and $g=1_B$ where $A,B$ are bounded open sets. Using \eqref{eq:covariancefirtformulaprelim} and that $\cX_L(f) = \cX(f_L)$ with $f_L(\cdot) = f(L^{-1}\cdot)$, we have the formula
\begin{align*}
L^{-d+1}\cov(\cX_L(f),\cX_L(g)) &= -\frac12 L\int\!\!\int (f(x+\frac{z}{L})-f(x))(g(x+\frac{z}{L})-g(x)) \, \bK(\md z) \, \md x \  \\
&= -\frac12 \int h_{\frac{1}{L}}(z)\bK(\md z)
\end{align*}
where $h_s$ is as defined in \eqref{defn:hs}.
%
By Lemma~\ref{lem:geometricboundary}, as $s\downarrow 0$, we have $h_s(z)\to h_0(z)$ for each $z\in \R^d$.
Further, since $|(A-sz) \setminus A| \leq Cs|z|$ for small enough $s$ and similarly for $B$, we have that $|h_s(z)| \leq C|z|$ for $s$ small. From the assumption that $\int |z||\bK|(\md z)<\infty$, DCT applies to give
\[
L^{-d+1}\cov(\cX(f_L),\cX(g_L))  \to -\frac12 \int  h_0(z) \ \bK(\md z),
\]
which completes the proof.
\end{proof}

\begin{proof}[Proof of Lemma~\ref{lem:geometricboundary}]  We shall break the proof into four steps for convenience.  
Fix  $C^1$ domains $A,B$ and $z \in \R^d$. Let $N_A(\cdot)$ and $N_B(\cdot)$ denote the respective outward unit normals. Let $E = \partial A \cap \partial B$.
\begin{enumerate}[wide, labelwidth=!, labelindent=0pt]
\item[\underline{STEP 1.}](\dy{{\em Common chart for $\partial A$ and $\partial B$}}) Given $p\in E$, we can find a common chart for both $\partial A$ and $\partial B$ in a neighbourhood of $p$. By this we mean that there exist $r_p>0$, an orthonormal matrix $M_p$, an open set $U_p \subseteq \R^{d-1}$ and  $C^1$ functions $\phi_p,\psi_p:U_p\to \R$, and $\eps_p\in \{+1,-1\}$, such that 
\begin{align}
\label{eq:commonchart} A\cap B(p,r_p) &=\{x\in B(p,r_p) \;  : \; t>\phi_p(y) \} \qquad 
B\cap B(p,r_p) =\{x\in B(p,r_p) \;  : \; \eps_p t > \eps_p \psi_p(y) \},
\end{align}
where $M_p(x-p)=(y,t)\in \R^{d-1}\times \R$. For $A$ alone, the existence of a chart is the definition of $C^1$-smoothness. But we have the freedom to choose the ``vertical" $t$-axis to be any direction that makes an angle of less than $\frac{\pi}{2}$ with $N_A(p)$. Choose the sign $\eps_p$ so that $\eps_pN_B(p)$ makes an angle of less than $\pi$ with $N_A(p)$. Then we can choose a $t$-axis that makes an angle of less than $\frac{\pi}{2}$ with $N_A(p)$ and with $\eps_p N_B(p)$. Thus the existence of a common chart. 

\item[\underline{STEP 2.}](\dy{{\em Partition of unity to decompose $h_s$ }})  By compactness of $E$, we can pick a finite set $F\subseteq E$ such that $E':=\cup_{p\in F}B(p,\frac12 r_p)$ covers  $E$. Choose a partition of unity $\{v_p\ : \ p\in F\}$ subordinate to $\{B(p,r_p)\ : \ p\in F\}$. Thus $v_p$ is a  smooth function with compact support inside $B(p,r_p)$ and such that $\sum_{p\in F} v_p=1$ on $E'$. Using this, we can write for any $z\in \R^d$ and small enough $s>0$,
\begin{align}\label{eq:hsusingpou}
h_s(z) 
&= \sum_{p\in F}\frac{1}{s}\int v_p(x) (1_{(A-sz)\setminus A}(x)-1_{A\setminus (A-sz)}(x))(1_{(B-sz)\setminus B}(x)-1_{B\setminus (B-sz)}(x)) \; \md x. 
\end{align}

\item[\underline{STEP 3.}](\dy{{\em Local asymptotics of $h_s$}}) In this step we fix $p\in F$ and  drop it in the subscripts for notational simplicity. We work in the rotated co-ordinates $x=(y,t)$ so that $A\cap B(p,r)=\{(y,t): y\in U, \  \phi(y)<t\}$ and $B\cap B(p,r)=\{(y,t) : y\in U, \  \eps \psi(y)<\eps t\}$. Partition  $U$ into
\begin{align*}
G_0&=\{y\in U\ : \ \phi(y)=\psi(y)\mbox{ and }\nabla \phi(y)=\nabla \psi(y)\}, \\
G_1&=\{y\in U : \ \phi(y)\not=\psi(y)\}, \\
G_2&=\{y\in U : \ \phi(y)=\psi(y), \ \nabla \phi(y)\not=\nabla \psi(y)\}.
\end{align*}
Applying  the inverse function theorem to $\phi-\psi$, we see that $G_2$ is a manifold of dimension $d-2$, and hence has zero measure (under the $(d-1)$-dimensional Lebesgue measure on $U$).

Now consider the summand for $p$ in \eqref{eq:hsusingpou}. In $(y,t)$ co-ordinates, the integral is over $U\times \R$ and the integral is 
\begin{align*}
h_s^{p}(z)&:=\int_U H_s^p(y) \ \md y \; \mbox{ where}, \\
H_s^p(y)&=\frac{1}{s}\int_{\R} v_p(y,t) (1_{(A-sz)\setminus A}(y,t)-1_{A\setminus (A-sz)}(y,t))(1_{(B-sz)\setminus B}(y,t)-1_{B\setminus (B-sz)}(y,t)) \; \md t.
\end{align*}
Writing $z=(w,r)$ in the $(y,t)$-coordinates. For  $y\in U$, and $t\in \R$, we have:
\begin{enumerate}
\item $(y,t)\in (A-sz)\setminus A$ if and only if $\phi(y+sw) - sr < t \le \phi(y)$. For small $s$, this is an interval of length $s(r-\nabla \phi(y).w)_+ + o(s)$.
\item $(y,t)\in A\setminus (A-sz)$ if and only if $\phi(y+sw) - sr \ge t > \phi(y)$. Up to $o(s)$ terms, this is an interval of length $s(r-\nabla \phi(y).w)_-$.
\item $(y,t)\in (B-sz)\setminus B$ if and only if $\eps (\psi(y+sw) - sr) < \eps t \le \eps \psi(y)$. Up to $o(s)$ terms, this  interval has length $s(r-\nabla \psi(y).w)_+$ for $\eps=1$ and length  $s(r-\nabla \psi(y).w)_- $ for $\eps=-1$.
\item $(y,t)\in B\setminus (B-sz)$ if and only if $\eps(\psi(y+sw) - sr) \ge \eps t > \eps \psi(y)$. Up to $o(s)$ terms, this  interval has  length $s(r-\nabla \psi(y).w)_- $ if $\eps=1$ and length $s(r-\nabla \psi(y).w)_+$ if $\eps=-1$.
\end{enumerate}
The first two intervals are in the vicinity of $\phi(y)$ while the last two are in the vicinity of $\psi(y)$. Therefore, if  $y\in G_1$, then the  first two are disjoint from the last two, for small enough $s$. Hence, $H_s^p(y)\to 0$ for $y \in G_1$.

On the other hand, if $y\in G_0$, then when $\eps=1$, either the first and third intervals are the same (up to $o(s)$ terms) or the second and fourth are the same. When $\eps=-1$, either the first and fourth intervals are the same or the second and fourth intervals are the same. This implies that $H_s^p(y)\to \eps |r-\nabla \phi(y).w|$ for $y \in G_0$.

For all small $s$, it is clear that $H_s(y)\le C|z|$ (as $A\Delta (A-sz)$ has measure at most $Cs|z|$) for all $y \in U$. Therefore, applying DCT (we may ignore $G_2$ as it has zero measure), we get
\begin{align}\label{eq:limtofhsp}
h_s^p(z)\to \eps \int_{G_0}v_p(y,\phi(y)) |r-\nabla \phi(y).w| \ \md y.
\end{align}
Note that on $G_0$, we may  replace $\phi$ and $\nabla \phi$ by $\psi$ and $\nabla \psi$ respectively.  Now observe that 
$$ 
r - \nabla \phi_p(y) .w = (r,w).(1,-\nabla \phi_p(y)) = \sqrt{1 + |\nabla \phi_p(y)|^2}(z.N_A(x)), 
$$ 
where $x=(y,t)$. Combining this with \eqref{eq:limtofhsp},  we get (and we put the subscript back in)
\begin{align*}
h_s^p(z)\to \int_{U}v_p(y,\phi_p(y)) 1_{\phi_p(y)=\psi_p(y)} 1_{\nabla \phi_p(y)=\nabla \psi_p(y)}\eps_p |z.N_A(y,\phi_p(y))| \ \sqrt{1 + |\nabla \phi_p(y)|^2} \ \md y.
\end{align*}
Observe that for all  $y\in G_0$, we have $\eps(x)=\eps_p$, where $\eps(x)=\<N_A(x),N_B(x)\>$ and $x=(y,\phi(y))=(y,\psi(y))$. Further, as $G_2$ has zero measure, we may drop the second indicator. Thus,
\begin{align}
h_s^p(z)&\to \int_{U_p}v_p(y,\phi_p(y)) 1_{\phi_p(y)=\psi_p(y)} \eps(y,\phi(y)) |z.N_A(y,\phi_p(y))| \  \sqrt{1 + |\nabla \phi_p(y)|^2}  \  \md y \nonumber \\
&= \int_{\partial A\cap B(p,r_p)}v_p(x)\eps(x) |z.N_A(x)| \ \sigma_A(\md x). \label{eq:limitofhspz}
\end{align}
As already remarked, on $G_0$ the functions $\phi_p$ and $\psi_p$ and their gradients coincide, hence we may also write the above equation with $B$ in place of $A$.

\item[\underline{STEP 4.}] (\dy{{\em Combining the local asymptotics}}) Sum \eqref{eq:limitofhspz} over $p\in F$ and recall that $\sum_p v_p(x)=1$ for $x\in E$, to get
\[
h_s(z)\to \int_{E} \eps(x) |z.N_A(x)| \sigma_A(\md x).
\]
This is precisely $h_0(z)$, completing the proof of the Lemma. \hfill \qedhere
\end{enumerate}
\end{proof}

\subsection{Variance lower bounds}
\label{s:varlb}

In this section,  we derive a lower bound for variance of stationary random measures with a non-trivial atomic part and truncated two-point correlation function in $d \geq 2$. By the Riemann-Lebesgue lemma, the existence of a truncated two-point correlation function implies that $\widehat{\bK}$ vanishes at $\infty$ and we need this assumption. 
\begin{proposition}
\label{prop:var_lb_stat_rand_meas}
Let $\cX$ be a stationary random measure on $\R^d, d \geq 2$ whose truncated correlation measure $\bK$ has finite total variation and $\widehat{\bK}$ vanishes at $\infty$. Let $\cX_L$ be the scaled random measure defined by $\cX_L(A) := \cX(LA)$,  $L \geq 1$.  Then,  there exists a constant $c >0$ (depending only on $d$) such that for any bounded set $W$ and for all $L \geq 1$,  
\begin{equation}
\label{e:lb_var}
\Var{\cX_L(W)} \geq c\lambda_D \min \{|W|L^d,  |W|^{\frac{d-1}{d}}L^{d-1}\},
\end{equation}
where $| \cdot |$ denotes the Lebesgue measure.
\end{proposition}
This proposition reduces to  that of Nazarov and Sodin \cite[Lemma 1.6]{Nazarov12} if we assume that $\cX$ is a point process in $\R^2$. Similar bounds were derived by Beck \cite{Beck1987}; see also \cite{Bjorklund2024}. However,  essentially following the proof of \cite[Lemma 1.6]{Nazarov12}, we obtain a non-trivial variance lower bound for stationary random measures with a non-trivial atomic part i.e., $\lambda_D > 0$. This is more widely applicable for various stochastic geometric models as explained below \eqref{e:inducedrm}. \dy{Although the lower bound on the right uses only the atomic part of the random measure, we are {\em not} claiming that if $\cX'$ is the atomic part of $\cX$, then $\Var{\cX_L(W)}\ge \Var{\cX'_{L}(W)}$. In other words, the constant $c$ could be smaller than $1$. As the proof shows, on the Fourier side, the atomic part dominates the diffusive part at large frequencies. }

Combining the above proposition with Theorem \ref{t:covindfn}, we have that if $\lambda_D + \int \bK = 0$ and $\int |z||\bK|(\md z) < \infty$, then 
\begin{equation}
\label{e:varasymp_surface}
L^{-d+1}\Var{\cX_L(A)}  \to  -\frac12 \int_{\partial A} \left[\int |z\cdot N_A(x)| \ \bK(\md z) \, \right] \sigma(\md x) \in (0,\infty),
\end{equation}
where $A$ is a bounded open set with $C^1$-smooth boundary. In case $\cX$ is a stationary point process, the above claim \eqref{e:varasymp_surface} follows from \cite[Proposition 2]{Martin80} and \cite[Lemma 1.6]{Nazarov12}. In special cases such as Ginibre point process or zeros of GAF, the limit \eqref{e:varasymp_surface} is known to be positive; see \cite[Section 3.4]{Forrester2014} and \cite{Forrester1999} respectively. The above variance convergence also holds for perturbed lattices (see \cite[Theorem 5]{Yakir21} and \cite{Gacs1975}) even though our integrability assumption on the truncated two-point correlations does not hold for this case. Also note that without the integrability assumption, variance asymptotics of indicator statistics can be infinite for some sets $A$; see \cite{Kim2017} and \cite[Section 3.1]{Yakir21}. 

\dy{Consider purely atomic stationary random measures i.e., those for which $\X_{\text{diff}} = 0$ in \eqref{e:atomic1}. Their first two correlation functions are given in \eqref{e:alpha1_puratom} and \eqref{e:alpha2red_puratom} respectively. If the random measure is non-degenerate i.e., the intensity $\lambda$ and diagonal intensity $\lambda_D$ are both non-zero and finite. In such cases under assumptions on $\bK$, we have that $\liminf_{L \to \infty}L^{-d+1}\Var{\cX_L(W)} > 0$. Such a variance lower bound is often necessary for central limit theorems; for example, see \cite[Theorem 7.2]{BYY2022}.}

\dy{We now give a simple illustrative example of a purely atomic stationary random measure with the above variance lower bound and thus ensuring the central limit theorem for the random measure. Let $\P$ be the Ginibre process or zeros of G.E.F. as in Section \ref{s:examples}. Consider the random measures}
%
\begin{equation}
\label{e:inducedrm}
\cX^{ec}  := \sum_{X \in \P} \sum_{Y \in \P}\1{[|X-Y| \leq 1]} \delta_{X}\, \, ;  \, \,  \cX^{el}  := \sum_{X \in \P} \sum_{Y \in \P} |X -Y|\1{[|X-Y| \leq 1]} \delta_{X},
\end{equation}
\dy{representing the edge count and edge length processes respectively, associated to the random geometric graph with connection distance $1$ on $\P$ i.e., $X,Y \in \P$ have an edge if $|X-Y| \leq 1$. The 'locality' of the statistic inducing the random measure can be exploited to show that the truncated correlation measure $\bK$ is integrable and hence $\hat{\bK}$ vanishes at $\infty$; see \cite[Theorem 1.11 and Section 2.2]{BYY2019}. Thus, we have that both $\cX^{ec}$ and $\cX^{el}$ satisfy surface-order variance asymptotics \eqref{e:varasymp_surface}. Furthermore, $\cX^{ec}_L(W)$ and $\cX^{el}_L(W)$ satisfy central limit theorems via \cite[Theorem 1.14]{BYY2019}. Apart from the above variance lower bound, \cite[Theorem 1.14]{BYY2019} also requires a moment condition. Verification of the same for $\P$ being Ginibre process or zeros of G.E.F. has been indicated in \cite[Section 2.1]{BYY2019}.}

\dy{The above examples of random measures were chosen for their simplicity. More examples of such random measures and central limit theorems for the same, assuming suitable variance lower bounds can be found in \cite{BYY2019,BYY2022,Fenzl2019,Chen2021,Cong2023}. The required variance lower bound in such central limit theorems follows from \eqref{e:lb_var} and further \eqref{e:varasymp_surface} guarantees the positivity of the limit in variance asymptotics for $\var[\cX_L(W)]$.}

Turning to the proof, the plan is to first derive a variance lower bound for $ \cX(f)$ when $f \in L^1 \cap L^2$ (see \cite[Lemma 1.7]{Nazarov12}) with the lower bound given in terms of \dy{`high-frequencies' of} $\hat{f}$.  The proof is concluded by analyzing 'high-frequencies' of $\hat{1}_{LW}$.

\begin{proof}
We shall assume $W$ has positive Lebesgue measure as the claim is trivial otherwise. Let $f \in L^1 \cap L^2$.   Recall that $f_L(\cdot) = f(L^{-1}\cdot)$ and so we have that $\widehat{f}_L(y) = L^d\hat{f}(Ly)$.  Then, substituting this in \eqref{eq:prelimcovformulafourier} and noting that $\cX_L(f) = \cX(f_L)$,  we have
$$ \Var{\cX_L(f)} = \frac{1}{(2\pi)^d} \int |\hat{f}_L(t)|^2 (\widehat{\bK}(t) + \lambda_D)\md t.$$
Since the LHS exists and is non-negative for all $f \in L^1 \cap L^2$,  we have that $\widehat{\bK}(y) + \lambda_D \geq 0$ for all $y \in \mR^d$. As $\widehat{\bK}(y) \to 0$ for $|y| \to \infty$, we can choose $C$ such that $\widehat{\bK}(y) \geq \frac{-\lambda_D}{2}$ for $|y| \geq C$ and thus we derive that for a finite constant $c$,  
\begin{eqnarray}
\Var{\cX_L(f)} & \geq &  \frac{\lambda_D L^{2d}}{2(2\pi)^d} \int_{|y| \geq C}  |\widehat{f}(Ly)|^2 \md y \no \\
\label{e:lb_ftf} & = & c \lambda_D L^d\int_{|y| \geq CL} |\widehat{f}(y)|^2 \md y.
\end{eqnarray}
Assuming $\lambda_D > 0$ without loss of generality,  we will now proceed as in the proof of \cite[Proposition 1.6]{Nazarov12} to complete the proof.

Choose a non-negative multiplier $m \in C_0^{\infty}(\mR^d)$ such that $m \leq 1$, $m(\xi) = 0$ for $\|\xi\| \geq 2$ and $m \equiv 1$ in the unit ball. With $m_L(\xi) = m(L^{-1}\xi)$, observe that $\widehat{m}_L(x) = L^d\widehat{m}(Lx)$ and hence $\|\nabla \widehat{m}_L\|_{1} = L \|\nabla \widehat{m}\|_{1}$.  We define $\varphi_{W,L}$ by setting its Fourier transform to be $\widehat{\varphi}_{W,L} = \widehat{\1}_W \cdot m_L$.  Thus,  we have that $\varphi_{W,L}$ is smooth and further we obtain that $|\nabla \varphi_{W,L}| \leq C'L$ for some constant $C'$.   

Now setting $f = \1_W$ in \eqref{e:lb_ftf} and $R = CL$, we derive that
\begin{eqnarray}
\Var{\cX_L(W)} & \geq &  \frac{cL^d}{2} \int_{|y| \geq R} |\widehat{\1}_W(y)|^2 \md y \, \geq \, \frac{cL^d}{2} \int_{\mR^d} |\widehat{\1}_W(y)|^2(1 - m_R(y))^2 \md y \no \\
& =  & \frac{cL^d}{2} \int_{\mR^d}(\1_W - \varphi_{W,R}(y))^2 \md y. \label{e:lb_varchiLW}
\end{eqnarray}
We first consider the case $|\{\varphi_{W,R} \geq \frac{1}{2}\}| \geq \frac{1}{2}|W|.$  Let $\omega_d$ denote the volume of the unit ball in $\R^d$.
\begin{eqnarray}
\int_{\mR^d}(\1_W - \varphi_{W,R}(y))^2 \md y & \geq  & \int_{\varphi_{W,R} \leq \frac{1}{2}} (\1_W - \varphi_{W,R}(y))^2 \md y \geq \frac{1}{4}  \int_{\varphi_{W,R} \leq \frac{1}{2}}  \md y \geq \frac{1}{4C'R} \int_{\varphi_{W,R} \leq \frac{1}{2}} |\nabla \varphi_{W,R}| \no \\
& =  & \frac{1}{4C'R} \int_0^{\frac{1}{2}} \cH_{d-1}(\{\varphi_{W,R} = t\}) \md t \, \, \, \, \mbox{(using coarea formula)} \no \\
& \geq & \frac{d\omega_d^{(d-1)/d}}{4C'R} \int_0^{\frac{1}{2}} | \{\varphi_{W,R} \geq t\}|^{\frac{d-1}{d}} \md t \, \, \, \, \mbox{(using isoperimetric inequality)} \no \\
& \geq &  \frac{d\omega_d^{(d-1)/d}}{8C'R} |\{\varphi_{W,R} \geq \frac{1}{2} \}|^{\frac{d-1}{d}}  \no \\
& \geq &  \frac{d}{8C'R}(\frac{\omega_d}{2})^{(d-1)/d} |W|^{\frac{d-1}{d}} = \frac{d}{8C'CL}(\frac{\omega_d}{2})^{(d-1)/d} |W|^{\frac{d-1}{d}},
\end{eqnarray}
where we have used the coarea formula \cite[Lemma 3.2.11]{Federer1996} and isoperimetric inequality between Lebesgue volume $|\cdot|$ and the $(d-1)$-dimensional Hausdorff measure $\cH_{d-1}$.  Usually,  the isoperimetric inequality is between the Lebesgue volume and Minkowski content but we have used $\cH_{d-1}$ instead of the Minkowski content.  This is justified because of \cite[Remark 3.2.44]{Federer1996} and that level sets of $\varphi$ are smooth for a.e.  $t \in [0,1]$ due to the Morse-Sard theorem. 
 
The other case  of $|\{\varphi_{W,R} \geq \frac{1}{2}\}| \leq \frac{1}{2}|W|$ can be handled more easily and we can derive the bound that
\begin{align*}
\int_{\mR^d}(\1_W - \varphi_{W,R}(y))^2 \md y & \geq  \int_{\varphi_{W,R} \, \leq \, \frac{1}{2}} (\1_W - \varphi_{W,R}(y))^2 \md y  \, \geq \, \frac{1}{4} \int_{\varphi_{W,R} \, \leq \, \frac{1}{2}} \1_W  \md y \\
& \geq c |W \cap \{\varphi_{W,R} \leq \frac{1}{2}\}| \, \geq \, c \left( |W| -  | \{ \varphi_{W,R} \, \geq \, \frac{1}{2} \}| \right) \geq c|W|/2.
\end{align*}
Substituting these two bounds in \eqref{e:lb_varchiLW} completes the proof of \eqref{e:lb_var}.
\end{proof}

\section{Central limit theorems for statistics of point processes}
\label{s:clt}

In this section, we upgrade our variance asymptotics to central limit theorems for \dy{linear} statistics of \dy{hyperuniform} point processes under some stronger assumptions. In Section \ref{s:clts}, we introduce some necessary notation first and then state the two central limit theorems - one for Sobolev statistics and another for H\"{o}lder statistics. The central limit theorems are proved under certain assumptions on integrability of truncated correlation functions as well as an integral identity relating higher-order truncated correlation functions with lower order ones. In Section \ref{s:highertruncorr}, we discuss examples of  point processes satisfying these conditions. In Section \ref{s:multilinearcumulants}, we prepare for our proofs by proving a key reduction formula (Proposition \ref{l:cumuredu}) for multilinear cumulants that uses an idea due to  B\'{a}lint Vir\'{a}g and the first author. The proofs of the two central limit theorems are provided in Section \ref{s:proofs_clt}. The proofs are by the method of moments, or more precisely by showing that if the statistics are scaled so that the variance converges, then higher cumulants vanish. By a \dy{variant of Marcinkiewicz's theorem}, this shows normal convergence of the scaled statistics. 


\subsection{The central limit theorems} 
\label{s:clts}

Let $\cX$ be a simple  point process on $\R^d$ such that for every $k\ge 1$, the $k$-point correlation function $\rho_k$ exists (see \eqref{e:kcorr}) and is locally-finite. The {\em truncated correlations} $\rho_k^T$ are defined as 
\begin{align}\label{eq:defnoftruncatedcorrelations}
\rho_k^{T}(x_1,\ldots ,x_k) :=\sum_{\pi\in \Pi(k)}(-1)^{\ell(\pi)-1}(\ell(\pi)-1)!\prod_{j=1}^{\ell(\pi)}\rho_{|\pi_j|}(x_{\pi_j}),
\end{align}
where $\Pi(k)$ is the collection of all set partitions of $[k] := \{1,\ldots,k\}$ and for $\pi\in \Pi(k)$ having parts $\pi_1,\ldots ,\pi_{\ell}$, we write $\ell(\pi)=\ell$ (number of blocks). For $S\subseteq [k]$, we write $x_S$ for the vector $(x_i)_{i\in S}$ (say in natural order, but it does not matter as we only apply symmetric functions to it).  For instance, $\rho_1^T(x)=\rho_1(x)$ and $\rho_2^T(x_{[2]})=\rho_2(x_1,x_2)-\rho_1(x_1)\rho_1(x_2)$ which is equal to $\kappa(x-y)$ when $\cX$ is stationary. By the M\"{o}bius inversion formula on the lattice of set partitions, \eqref{eq:defnoftruncatedcorrelations} is equivalent to 
\begin{align}\label{eq:defnoftruncatedcorrelations2}
\rho_k(x_1,\ldots ,x_k) :=\sum_{\pi\in \Pi(k)}\prod_{j=1}^{\ell(\pi)}\rho_{|\pi_j|}^T(x_{\pi_j}).    
\end{align}
Thus, truncated correlations carry the same information as correlations. From the form of \eqref{eq:defnoftruncatedcorrelations} and \eqref{eq:defnoftruncatedcorrelations2}, it is clear that truncated correlations and correlations are related like cumulants and moments. This will be seen in \eqref{eq:kappaqmpm} where cumulants of linear statistics are expressed in terms of the truncated correlation functions. It is also easily seen in point processes on discrete spaces, see Appendix~\ref{s:app_discrete}.

\begin{remark}
\label{rem:corr_diag}
We are sticking to a simple point processes on $\R^d$, with correlation functions being densities of correlation measures w.r.t. Lebesgue measure. Therefore, we have not had to clarify how correlation and truncated correlation functions are defined when $x_i$ are not distinct. Apart from extending $\rho_k$ by continuity on diagonals when possible, another convention for simple point processes is to set $\rho_k(x_1,\ldots ,x_k)=0$ if $x_i=x_j$ for some $i\not=j$, and then the definition of $\rho_k^T$ remains as stated in \eqref{eq:defnoftruncatedcorrelations}. In cases of point processes like Ginibre or zeros of G.E.F., the two conventions agree. The latter convention is the right one, for example, when considering a simple point process on $\Z^d$ with correlation functions being densities w.r.t. counting measure on $\Z^d$. This definition ensures the validity of  \eqref{eq:kappaqmpm}, which expresses cumulants of linear statistics of the point process in terms of the truncated correlations functions. To not interrupt the flow, we relegate a brief discussion to Appendix~\ref{s:app_discrete}. 
\end{remark}
Now we state our non-trivial assumption on truncated correlations, \dy{which may be viewed as higher-order hyperuniformity conditions}. 
\begin{assumption}\label{asmpt:integrationofrhokT}
For all $k \geq 2$ \dy{and for a.e. $x_1,\ldots,x_{k-1}$, we assume that $\rho_k^T$ is absolutely integrable with respect to $x_k$ and satisfies} 
\begin{align}
\int_{\R^d} \rho_k^{T}(x_1,\ldots ,x_k)\md x_{k} & = -(k-1)\rho_{k-1}^T(x_1,\ldots ,x_{k-1}). \label{eq:integralsoftruncatedcorrelations0}
\end{align}
\end{assumption}
\dy{The assumption~\ref{asmpt:integrationofrhokT} is not satisfied by all point processes.  For example,  $\rho_k^T \equiv 0$ for all $k \geq 2$ for a Poisson process but $\rho_1^T \equiv \rho_1$ and so \eqref{eq:integralsoftruncatedcorrelations0} is violated for $k = 2$. However, for stationary $\X$ with intensity $\lambda$ and $k=2$, \eqref{eq:integralsoftruncatedcorrelations0} is the same as $\int \kappa(z)\md z = -\dy{\lambda}$, which is equivalent to hyperuniformity \cite{Coste2021}. For higher $k$, this assumption is naturally motivated by considering point processes in the discrete setting having a constant, finite number of points; see Appendix \ref{s:app_discrete}. In the next subsection, we give examples of infinite point processes for which \eqref{eq:integralsoftruncatedcorrelations0} does hold. The way we make  use of Assumption~\ref{asmpt:integrationofrhokT} will be to get convenient form of cumulants of linear statistics of the point process. This will be done later in Proposition~\ref{l:cumuredu}, but for now let us illustrate the idea for the second cumulant, which is just the variance. If $f$ is a nice enough test function, then by \eqref{eq:integralsoftruncatedcorrelations0},}
\begin{align*}
\int\!\!\!\int f^2(x)\rho_2^T(x,y) \ \md y  \md x & = \int\!\!\!\int f^2(y)\rho_2^T(x,y) \ \md x  \md y  = - \int f^2(z)\rho_1(z)\ \md z , 
\end{align*}
from which and \eqref{eq:covariancefirtformulaprelim}, we easily obtain that
\begin{align}
    \textrm{Var}(\cX(f)) & =  \int f^2(z)\rho_1(z)\md z + \int\!\!\!\int f(x)f(y)\rho_2^T(x,y)\md y \ \md x \nonumber \\
    &=  -\frac12 \int \!\!\!  \int(f(x)-f(y))^2 \ \rho_2^T(x,y) \md x \ \md y. \label{eq:reductionformulaformequal2}
\end{align}
\dy{To see the usefulness of this formula, apply it to $\cX_L$ and let $L\to \infty$. Typically $\rho_{2,L}^T(x,y)$ concentrates in a shrinking neighbourhood around the diagonal line $\{x=y\}$ (because $\cX$ restricted to far off regions are nearly independent). But $f(x)-f(y)$ vanishes on the diagonal. Thus the above formula  explains the significantly smaller variance (i.e., hyperuniformity) when \eqref{eq:integralsoftruncatedcorrelations0} holds for $k=2$.}    

For later use, we note that by iterating  \eqref{eq:integralsoftruncatedcorrelations0},  for all $k>r \geq 1$, we get 
\begin{align}\label{eq:integralsoftruncatedcorrelations}
\int_{(\R^d)^{k-r}}\rho_{k}^T(x_1,\ldots ,x_k)\md x_{r+1}\ldots \md x_{k} = (-1)^{k-r}\frac{(k-1)!}{(r-1)!}\ \rho_r^T(x_1,\ldots ,x_r).
\end{align}

If $\X$ is stationary, we define $\kappa_{m-1}(z_1,\ldots ,z_{m-1}) := \rho_m^T(x+z_1,\ldots ,x+z_{m-1},x)$ for any $x$. For any   $\tau: [m] \to [m-1]$ and $\alpha^{(j)}\in \Z_+^d$, $j\le m$,   we define
\begin{equation}
\label{eq:defIm} I_m(\tau; \alpha^{(1)},\ldots ,\alpha^{(m)}) := \frac{1}{\alpha^{(1)}!\times \ldots \times \alpha^{(m)}!}\int_{(\R^d)^{m-1}}  \kappa_{m-1}(z_1,\ldots ,z_{m-1}) \prod\limits_{i=1}^m z_{\tau(i)}^{\alpha^{(i)}}\; \md z_1\ldots \md z_{m-1},
\end{equation}
provided the integral exists. 
\remove{Further, set  
\begin{align}
d_m &:= \min\left\{ \sum_{j=1}^m |\alpha^{(j)}| \; : \; \alpha^{(j)} \neq 0,  \; I_m(\tau; \alpha^{(1)},\ldots ,\alpha^{(m)}) > 0 \mbox{ for some }\tau:[m]\to [m-1]\right\} \no \\
\label{eq:defdm} & = \dy{1 + \max \left\{ \sum_{j=1}^m |\alpha^{(j)}| \; : \; \alpha^{(j)} \neq 0,  \; I_m(\tau; \alpha^{(1)},\ldots ,\alpha^{(m)}) = 0 \mbox{ for all }\tau:[m]\to [m-1]\right\} }  
\end{align}
\footnote{\dy{max definition better than min definition ? Lesser integrability assumption ?}}The definition implies that $d_m\ge m$. \dy{Evenness of $\kappa$ gives that $d_2$ is even as observed in Section \ref{s:covsmooth}.}}
We first give a criterion for asymptotic normality of Sobolev linear statistics of $\cX$. Recall definition of $I(\alpha)$'s and $Q_m$'s from Section \ref{s:exactformulas}.
\begin{theorem}(CLT for Sobolev statistics)
\label{thm:cltsmooth}
Let $\cX$ be a simple stationary point process on $\R^d$ with truncated correlations satisfying Assumption \ref{asmpt:integrationofrhokT}.  Assume that for all $m \geq 2$, there exists $d_m \geq m$ such that for all $l_1,\ldots,l_{m-1} \in \Z_+$ with $\sum_i l_i \leq d_m$,
\begin{align}\label{eq:intgrconditionforCLTSob}
    \int_{(\R^d)^{m-1}}  \prod\limits_{i=1}^{m-1} |z_{i}|^{l_i}\; |\kappa_{m-1}|(z_1,\ldots ,z_{m-1}) \, \md z_1\ldots \md z_{m-1} < \infty,
\end{align} 
and further that for all non-zero $\alpha^{(j)} \in \Z_+^d$, $1\le j\le m$, satisfying $\sum_{j=1}^m |\alpha^{(j)}| < d_m$, it holds that 
\begin{align}
\label{eq:defdm}
 I_m(\tau; \alpha^{(1)},\ldots ,\alpha^{(m)}) = 0, \, \, \forall \tau : [m] \to [m-1].   
\end{align}
Let $f \in L^1(\R^d)\cap H^k$  for some $k\ge 1$ (with the convention that $H^0=L^2(\R^d)$). Let $\lambda_D + I(0) = 0$ and $I(\alpha)= 0$ for  $1\le |\alpha|< 2k$. If $2(d-d_m)< m(d-2k)$ for $m$ large, then 
\[
\frac{\cX_L(f)-\E[\cX_L(f)]}{L^{(d-2k)/2}}\stackrel{d}{\longrightarrow} N(0,Q_{2k}^{\kappa}(f,f))
\]
\end{theorem}
\begin{remark} 
\label{rem:clts}
\begin{enumerate}
\item \mk{When $d_m=m$, the condition~\eqref{eq:defdm} is trivially satisfied as there are no non-zero $\alpha^{(j)}$ satisfying $\sum_{j=1}^m |\alpha^{(j)}| < m$. This happens when $2k < d + 2$, because then $2(d-m)< m(d-2k)$ for all large $m$. In such cases, we only need to check the integrability condition \eqref{eq:intgrconditionforCLTSob}.  In particular, this applies to the Ginibre process with $d = 2$, $k=1$, $Q_{2}^{\kappa}(f,f) = \frac{1}{4\pi}\|\nabla f\|_2^2$ (this was shown in Section \ref{s:examples}). See Example~\ref{ex:dppclt} below for full  details and generalization to stationary determinantal processes. }


\item By polarization, one can rephrase the result as saying that the random field $f\mapsto \frac{\cX_L(f)-\E[\cX_L(f)]}{L^{(d-2k)/2}}$ indexed by $L^1\cap H^k$ functions converges, in the sense of finite dimensional distributions, to the centered Gaussian field with covariance $Q_{2k}^{\kappa}(f,g)$. \mk{Observe that if we do not take the maximum possible value of $k$, the result holds but the limiting variance is zero.}

\item If the interest is in a particular function $f$, one may want to say that $\frac{\cX_L(f) - \E[\cX_L(f)]}{\sqrt{\var[\cX_L(f)]}}\stackrel{d}{\rightarrow} N(0,1)$. The theorem does imply this, provided $f$ satisfies the assumptions of the Theorem and in addition we have $Q^{\kappa}_{2k}(f,f)\not=0$. This is because   Proposition~\ref{l:covasymstatfindiffsobolev} implies that $\var[\cX(f_L)] \sim Q_{2k}^{\bK}(f,f) L^{d-2k}$.  

\item When $\cX$ is the zero set of the  G.E.F., we expect the above CLT to hold with $k=2$ and $Q_{4}^{\kappa}(f,f) = \frac{\zeta(3)}{16\pi}\|\Delta f\|_2^2$. However, we are unable to verify \eqref{eq:integralsoftruncatedcorrelations0} but we do show that \eqref{eq:intgrconditionforCLTSob} and \eqref{eq:defdm} hold for some $d_m \geq 2m - 2$. See Section \ref{s:highertruncorr} for details. 

\item As will be seen from the proof, the condition  $2(d-d_m)<m(d-2k)$ is needed to prove the CLT for compactly supported smooth functions and then we extend it to functions in $L^1\cap H^k$. If one were to find an alternate condition that guarantees the CLT for compactly supported smooth functions, then our proof would readily extend it to functions in $L^1\cap H^k$ by arguments as in Proposition~\ref{l:covasymstatfindiffsobolev}.

\end{enumerate}  
\end{remark}
Next we state a CLT for H\"{o}lder continuous statistics of stationary point processes satisfying the integrability condition for truncated correlations and an appropriate variance lower bound. For $\alpha \in (0,1]$, let $C_0^{\alpha}$ be the space of compactly supported $\alpha$-H\"{o}lder continuous functions on $\R^d$. We also write $C_0^0$ for the space of bounded functions on $\R^d$.
\begin{theorem}(CLT for H\"{o}lder continuous statistics)
\label{thm:cltholder}
Let $\alpha \in [0,1]$. Let $\cX$ be a simple  stationary point process on $\R^d$ satisfying Assumption \ref{asmpt:integrationofrhokT}. Assume that for all $m \geq 2$ and for all $l_i \in \Z_+, \sum_{i=1}^{m-1} l_i= m$, 
\begin{equation}
 \label{eq:intgrconditionforCLTHol}   
 \int_{(\R^d)^{m-1}}  \prod\limits_{i=1}^{m-1} |z_{i}|^{\alpha l_i}\; |\kappa_{m-1}|(z_1,\ldots ,z_{m-1}) \, \md z_1\ldots \md z_{m-1} < \infty.
 \end{equation}
Suppose that $f \in C_0^{\alpha}$ and $\var[\cX_L(f)] \geq CL^{-2\alpha + \eps}$ for some $\eps,C > 0$.  Then CLT holds for $\cX_L(f)$, i.e.,
\[  \frac{\cX_L(f) - \E[\cX_L(f)]}{\sqrt{\var[\cX_L(f)]}} \stackrel{d}{\longrightarrow} N(0,1).\]
\end{theorem}
This is analogous to the main asymptotic normality result of \cite[Theorem 1.2]{NazarovSodin11} which was  stated  for zeros of the Gaussian entire function, and  proved by a method specific to that point process. We use entirely different and general methods, but we cannot yet claim that Theorem~\ref{thm:cltholder} is a generalization of \cite[Theorem 1.2]{NazarovSodin11}, as it remains to verify condition~\eqref{eq:integralsoftruncatedcorrelations0} for zeros of the Gaussian entire function, see Example~\ref{ex:GEFzerostrunccorr} below. The variance lower bound assumption need not hold always; see the remark after Corollary~1.3 in  \cite{NazarovSodin11}.  

\dy{The proofs of both the central limit theorems are via cumulant method wherein it suffices to show vanishing of higher cumulants. The new input here is the cumulant reduction formula in Proposition~\ref{l:cumuredu}, which yields improved estimates on cumulants as suggested below Assumption \ref{asmpt:integrationofrhokT}.}

\subsection{Examples of point processes.}
\label{s:highertruncorr}

We now provide examples to which the central limit theorems \ref{thm:cltsmooth} and \ref{thm:cltholder} apply. Mainly, we shall verify the non-trivial integral identity \eqref{eq:integralsoftruncatedcorrelations0} and the integrability conditions \eqref{eq:intgrconditionforCLTSob}, \eqref{eq:intgrconditionforCLTHol} for certain Determinantal point processes with a projection kernel and partially for  zeros of G.E.F..  Although all point processes in this paper are infinite and even stationary, the cleanest way to understand the integral identity~\eqref{eq:integralsoftruncatedcorrelations0} is in the context of finite point processes and so we begin with this case.

\begin{example}(Finite processes) Let $\cX$ be a point process with a total of $n$  points (here $n$ is fixed, not random).     Adopt the convention in \eqref{eq:defnoftruncatedcorrelations} that the first part $\pi_1$ contains $x_{k}$.
\begin{align}
\int_{\R^d} \rho_k^{T}(x_1,\ldots ,x_k)\md x_{k} &= \sum_{\pi\in \Pi(k)}(-1)^{\ell(\pi)-1}(\ell(\pi)-1)!(n-|\pi_1|+1)\rho_{|\pi_1|-1}(x_{\pi_1\setminus\{k\}})\prod_{j=2}^{\ell(\pi)}\rho_{|\pi_j|}(x_{\pi_j}) \nonumber \\
&= \sum_{\sigma\in \Pi(k-1)}(-1)^{\ell(\sigma)-1}(\ell(\sigma)-1)!\prod_{j=1}^{\ell(\sigma)}\rho_{|\sigma_j|}(x_{\sigma_j})\left\{\sum_{j=1}^{\ell(\sigma)}(n-|\sigma_j|) \ - \ n \ell(\sigma) \right\} \nonumber \\
&= -(k-1)\rho_{k-1}^T(x_1,\ldots ,x_{k-1}).   \no
\end{align}
In the second equality, we used the fact that  $\sigma\in \Pi(k-1)$ can arise from many different $\pi\in \Pi(k)$.  One of them is the one having all the blocks of  $\sigma$ but in addition a singleton block $\{k\}$; in this case $\ell(\pi)=\ell(\sigma)+1$. Another $\ell(\sigma)$ many set-partitions $\pi$ arise by including $k$ in one of the blocks of $\sigma$; in these cases $\ell(\pi)=\ell(\sigma)$.   

As we think that this example conveys an important point, a different and arguably more transparent proof is given in Appendix~\ref{s:app_discrete} for the further special case of point processes with a fixed number of points on a countable space.
\end{example}
Suppose a simple point process $\cX$ can be ``well-approximated" by finite simple point processes $\cX_n$ having $\cX_n(R^d)=n$. By well-approximated, we mean that for any $k\ge 1$ and any $x_1\in \R^d$, the truncated correlation function $\rho_{n,k}^T(x_1,x_2,\ldots ,x_k)$ of $\cX_n$ converges in $L^1((\R^d)^{k-1})$ to the truncated correlation function $\rho_k^T(x_1,\ldots ,x_{k-1})$ of $\cX$.   Then $\cX$ also satisfies Assumption~\ref{asmpt:integrationofrhokT}.

This example suggests that the Assumption~\ref{asmpt:integrationofrhokT}
holds for those point processes that have a ``fixed total number" of points. For infinite processes, number rigidity in the sense of Ghosh and Peres~\cite{Ghosh2017} is a way of making precise such a conservation of numbers. Here we just give two such classes of examples.

\begin{example}(Determinantal process with a projection kernel) \label{ex:dppclt} Let $\cX$ be a determinantal process with kernel $K$. Then, it is known (see (4.9) in \cite{soshnikovsurvey}) that
\[
\rho_k^T(x_1,\ldots ,x_{k}) =(-1)^{k-1}\sum_{P\in \mathcal S_k^{\tiny cyc}}\prod_{i=1}^k K(x_i,x_{P(i)})
\]
where $\mathcal S_k^{\tiny cyc}$ denotes the set of cyclic permutations of $[k]$. It is not hard to see this by checking that \eqref{eq:defnoftruncatedcorrelations} or \eqref{eq:defnoftruncatedcorrelations2} holds \dy{and that correlation functions are given by determinants i.e., $\rho_k(x_1,\ldots ,x_{k}) = \textrm{det}\big( K(x_i,x_j)\big)_{1 \leq i,j \leq k}$.}

To verify Assumption~\ref{asmpt:integrationofrhokT}, we must and do assume that $K$ is a projection kernel, i.e., it is Hermitian ($K(x,y)=\overline{K(y,x)}$) and satisfies the reproducing  property $\int K(x,y)K(y,z)\md y=K(x,z)$. So, integrating the truncated correlation functions for such processes, we obtain that
\begin{align*}
\int \rho_k^T(x_1,\ldots ,x_k)\md x_k &= (-1)^{k-1}\sum_{P\in \mathcal S_k^{\tiny cyc}}\int \prod_{i=1}^k K(x_i,x_{P(i)}) \md x_k.
\end{align*}
For a cyclic $P$, there is some $i,j$ such that $P(i)=k$ and $P(k)=j$. Let $Q\in \mathcal S_{k-1}^{\tiny cyc}$ be defined by $Q(i)=j$ and $Q(i')=P(i')$ for all $i'\in [k-1]\setminus\{i\}$. By the reproducing kernel property 
\begin{align*}
\int \prod_{i=1}^k K(x_i,x_{P(i)}) \md x_k = \prod_{i=1}^{k-1}K(x_i,x_{Q(i)}).    
\end{align*}
Further, each $Q\in \mathcal S_{k-1}^{\tiny cyc}$ arises from $k-1$ different $P\in \mathcal S_{k}^{\tiny cyc}$ (for a given $Q$, we can insert $k$ between $i$ and $Q(i)$ for any $i\in [k-1]$). Therefore, we see that
\begin{align*}
\int \rho_k^T(x_1,\ldots ,x_k)\md x_k &= (-1)^{k-1}(k-1)\sum_{Q\mbox{ \scriptsize{cyclic}}}\prod_{i=1}^{k-1} K(x_i,x_{Q(i)}) \\
&= -(k-1)\rho_{k-1}^T(x_1,\ldots ,x_{k-1})
\end{align*}
Thus Assumption~\ref{asmpt:integrationofrhokT} is satisfied. Further, if $|K(z,w)|$ decays faster than any polynomial in $|z-w|$ as $|z-w| \to \infty$, it follows that the condition~\eqref{eq:intgrconditionforCLTSob} and \eqref{eq:intgrconditionforCLTHol} are satisfied. \mk{As explained in Remark~\ref{rem:clts}, if $I(\alpha)=0$ for $|\alpha|< 2k$ where $2k<d+2$, then the condition~\eqref{eq:defdm} is trivially satisfied with $d_m=m$. For  stationary determinantal processes, $k=1$ works i.e., $I(\alpha) = 0$ when $|\alpha|=1$. This is because  $\kappa_2(x)$ is an even function while $x^{\alpha}$ is an odd function when $|\alpha|=1$.}

\mk{Thus the CLT for Sobolev statistics (i.e., Theorem \ref{thm:cltsmooth}) holds for stationary determinantal processes with a projection kernel that decays fast off the diagonal. If $2k < d$ then the variance grows faster than a power of the mean, and the CLT result is then known even without explicit decay condition on the kernel, see \cite{soshnikovclt}. }

To deduce asymptotic normality for H\"{o}lder statistics as in Theorem \ref{thm:cltholder}, the variance lower bound condition alone needs to be verified.
\end{example}

\begin{example}(Ginibre Ensemble.) For the infinite  Ginibre ensemble, we have $d=2$ and from the derivation in Section \ref{s:examples}, we also have that $d_2 = 2$ with $Q_{2}^{\kappa}(f,f) = \frac{1}{4\pi}\|\nabla f\|_2^2\neq 0$. Thus from Remark \ref{rem:clts}(1), $2(d-d_m) < m(d -2k)$ holds for $k = 1$. From the determinantal structure and from the exponential decay of the kernel $|K(z,w)|^2=e^{-|z-w|^2}$, it follows that Assumption~\ref{asmpt:integrationofrhokT} and the conditions~\eqref{eq:intgrconditionforCLTSob} and \eqref{eq:intgrconditionforCLTHol} are satisfied. Therefore we get asymptotic normality of Sobolev linear statistics of the infinite Ginibre ensemble as in Theorem \ref{thm:cltsmooth} with $k = 1$ 
\end{example}

\begin{example}(Zeros of random analytic function) \label{ex:GEFzerostrunccorr} Refer to \cite{Ben09} for basic definitions and facts about random analytic functions, in particular Gaussian analytic functions (GAFs).  Let $\cX$ be the zero set of a random analytic function $F$ on a region $\Omega\subseteq \mathbb C$. Then $\cX=\frac{1}{2\pi}\Delta_z\log |F|$ and hence, for distinct $z_1,\ldots ,z_{k}$,
\[
\rho_k(z_1,\ldots ,z_k)=\frac{1}{(2\pi)^k}\Delta_{z_1}\ldots \Delta_{z_k}\E[\log|F(z_1)|\ldots \log|F(z_k)|].
\]  
Therefore,
\begin{align}
\rho_k^{T}(z_1,\ldots ,z_k) &=\frac{1}{(2\pi)^k}\sum_{\pi\in \Pi(m)}(-1)^{\ell(\pi)-1}(\ell(\pi)-1)!\prod_{j=1}^{\ell(\pi)}\left(\prod_{i\in \pi(j)}\Delta_{z_i}\right)\E\left[\prod_{i\in \pi(j)}\log|F(z_i)|\right] \nonumber \\
&=\frac{1}{(2\pi)^k}\Delta_{z_1}\ldots \Delta_{z_k}\sum_{\pi\in \Pi(m)}(-1)^{\ell(\pi)-1}(\ell(\pi)-1)!\prod_{j=1}^{\ell(\pi)}\E\left[\prod_{i\in \pi(j)}\log|F(z_i)|\right]  \nonumber \\
&= \frac{1}{(2\pi)^k}\Delta_{z_1}\ldots \Delta_{z_k}\bkappa_k[\log|F(z_1)|,\ldots ,\log|F(z_k)|] \label{eq:truncorrforGAF}
\end{align}
where $\bkappa_k$ is the multi-linear cumulant function of real-valued random variables (recalled in the subsection~\ref{s:multilinearcumulants}). Again, this formula is valid for distinct $z_1,\ldots ,z_k$. Singularities appear when $z_i$ and $z_j$ get close for some $i\not=j$. As the cumulant vanishes when one of the variables is a constant (and for $k=1$ there is the Laplacian which kills constants), we may also rewrite \eqref{eq:truncorrforGAF} as
\begin{align}\label{eq:truncorrforGAF2}
    \rho_k^{T}(z_1,\ldots ,z_k) &=\frac{1}{(2\pi)^k}\Delta_{z_1}\ldots \Delta_{z_k}\bkappa_k[\log|\widetilde{F}(z_1)|,\ldots ,\log|\widetilde{F}(z_k)|]
\end{align}
where $\widetilde{F}(z)=F(z)/\sqrt{\E[|F(z)|^2]}$ is not holomorphic, but normalized to have $\E[|\widetilde{F}(z)|^2]=1$. 

Is Assumption~\ref{asmpt:integrationofrhokT} valid? The integrability of $\rho_k^T(z_1,z_2,\ldots ,z_k)$ on $\C^{k-1}$ for fixed $z_1$ was proved by Nazarov and Sodin~\cite{Nazarov12} (Claims~4.1, 4.2 together with Theorem~1.4 in that paper) for very general Gaussian analytic functions, including the G.E.F. Thus, the integral on the left side of \eqref{eq:integralsoftruncatedcorrelations0} exists. Is the integral equal to the right side? For zeros of G.E.F. and $k=2$, \eqref{eq:kappagaf}  gives 
\begin{align*}
\rho_2^T(z_1,z_2)=-\frac{1}{\pi^2} \;  \sum_{m\ge 1}e^{- m|z_1-z_2|^2} (2-4m|z_1-z_2|^2+m^2|z_1-z_2|^4).
\end{align*}
Fixing $z_1$ and integrating over $z_2$ gives $\frac{-1}{\pi}$ which is $-\rho_1^T(z_1)$ as required. It is worth noting that interchange of integral and sum is not valid, in fact the sum of integrals is zero! 

For general $k\ge 3$, we have not been able to carry out the computations to show  that \eqref{eq:integralsoftruncatedcorrelations0} holds. 
We present some of our investigations in Appendix~\ref{app_GAFcumulants}, in particular a reduced identity \eqref{eq:reducedidentityforRAF} which would imply \eqref{eq:integralsoftruncatedcorrelations0}. For now we leave this as a question.

\smallskip
\noindent{\bf Question}: Is  \eqref{eq:reducedidentityforRAF}  and hence \eqref{eq:integralsoftruncatedcorrelations0} valid for  $k\ge 3$ for the G.E.F.? For  general Gaussian analytic functions? Even for non-Gaussian random analytic functions under mild conditions?

\end{example} 

\begin{example}(Zeros of the G.E.F.) In this case, $d=2$, $d_2= 2k = 4$ with $Q_{4}^{\kappa}(f,f) = C_2 \| \Delta f \|_2^2$. Hence Remark \ref{rem:clts}(5) cannot be used. We now claim that $d_m \ge 2m-2$ (proven in the next paragraph), which then shows that $2(d-d_m) < m(d-d_2)$ for all $m \ge 5$. The integrability conditions \eqref{eq:intgrconditionforCLTSob} and \eqref{eq:intgrconditionforCLTHol} follow from \cite[Theorem 1.4 and Claim 4.1]{Nazarov12} and the integral identity \eqref{eq:integralsoftruncatedcorrelations0} has been conjectured to be true in the previous example. Thus, assuming a positive answer to the question above, we have asymptotic normality of Sobolev statistics as in Theorem \ref{thm:cltsmooth} with $k=2$ and $Q_{4}^{\kappa}(f,f) = C_2 \| \Delta f \|_2^2$. Further, we also have asymptotic normality of H\"{o}lder linear statistics as in Theorem \ref{thm:cltholder}, since the necessary variance lower bound has been proven in \cite{NazarovSodin11}. We recall that these two CLTs have been proven via different methods in \cite[Theorem 1.2 and Corollary 1.3]{NazarovSodin11}.

To show that $d_m\ge 2m-2$,  write
\begin{align}
I_m(\tau; \alpha^{(1)},\ldots ,\alpha^{(m)}) &= C\int_{\mathbb C^{m-1}}  \kappa_{m-1}(z_1,\ldots ,z_{m-1}) \prod\limits_{i=1}^m z_{\tau(i)}^{\alpha^{(i)}}\; \md z_1\ldots \md z_{m-1}.   \label{eq:formulaforImforGEF}
\end{align}
Here $\alpha^{(i)}\in \Z_+^2$ for each $i$, and the natural identification of $\mathbb C$ with $\R^2$ is used in writing  $z_{\tau(j)}^{\alpha^{(j)}}$. The formula \eqref{eq:truncorrforGAF2} gives
$$
\kappa_{m-1}(z_1,\ldots ,z_{m-1})=(2\pi)^{-m}\Delta_{z_1}\ldots \Delta_{z_{m-1}}\Delta_{x}\bkappa_m[\log|\widetilde{F}(x+z_1)|,\ldots ,\log|\widetilde{F}(x+z_{m-1})|,\log|\widetilde{F}(x)|]
$$  
for any $x\in \mathbb C$. Plug this into \eqref{eq:formulaforImforGEF} and integrate by parts to transfer  the Laplacians $\Delta_{z_i}$ onto $\prod\limits_{i=1}^m z_{\tau(i)}^{\alpha^{(i)}}=\prod\limits_{k=1}^{m-1}z_k^{\beta^{(k)}}$ where $\beta^{(k)}$ is the sum of $\alpha^{(i)}$ over all $i$ such that $\tau(i)=k$.  If $|\alpha^{(1)}|+\ldots +|\alpha^{(m)}|<2m-2$, then $|\beta^{(k)}|<2$ for some $k\le m-1$, and then $\Delta_{z_{k}}$ annihilates $\prod\limits_{i=1}^m z_{\tau(i)}^{\alpha^{(i)}}$. Thus,  $I_m(\tau; \alpha^{(1)},\ldots ,\alpha^{(m)})=0$ if $|\alpha^{(1)}|+\ldots +|\alpha^{(m)}|<2m-2$. Therefore, $d_m\ge 2m-2$ \dy{and hence \eqref{eq:defdm} holds for zeros of G.E.F}. 
\end{example}

\begin{remark} The idea behind \eqref{eq:truncorrforGAF}  may be  useful in other situations. If we can represent our point process as $\cX=\cL F$ for a real-valued random field $F$ and a linear operator (e.g., a linear differential operator) $\cL$, then under mild conditions allowing interchange of $\cL$ and integrals,
\[
\rho_k^T(x_1,\ldots ,x_k)=\cL_{x_1}\ldots \cL_{x_k} \bkappa_k[F(x_1),\ldots ,F(x_k)].
\]
Here $\bkappa_k[F(x_1),\ldots ,F(x_k)]$ is the multi-linear cumulant of the random variables $F(x_i)$, $i\le k$. 

When can we write a given point process $\cX$ as $\mathcal L F(x)$?  A natural choice for processes on $\R^d$ is to define  $F(x)=\cX(G_d(\cdot,x))=\sum_{y\in \cX}G_d(y,x)$ where $G_d(y,x)$ is the Green's function on $\R^d$ and $\cL=\Delta$, the Laplacian. The catch is that there can be convergence issues for infinite point processes and additional challenges if we want $F$ to be stationary, a natural requirement if $\cX$ is stationary. Such a construction  for Poisson processes in $d\ge 5$ was carried out in \cite{Chatterjee2010} \dy{and for fairly general hyperuniform point processes (under assumptions weaker than those of Theorem \ref{t:covindfn}(2)) in \cite[Theorem 5.1]{sodinfluctuations1}. In particular, this includes the Ginibre process and zeros of G.E.F.}
\end{remark}

\subsection{Multilinear cumulants of linear statistics of a point process} 
\label{s:multilinearcumulants}
As explained in the introduction of this section, we prove the central limit theorems by deriving bounds for the cumulants of linear statistics of the point process. These cumulants can be expressed as integrals against the  truncated correlation functions. In getting upper bounds on cumulants, it becomes important to take advantage of the cancellation implicit in the signed sum in the definition~\eqref{eq:defnoftruncatedcorrelations}. One key ingredient of our proof is Proposition~\ref{l:cumuredu} that extracts this cancellation efficiently, when the integration condition \eqref{eq:integralsoftruncatedcorrelations0} is satisfied. This form of writing the cumulants of linear statistics  was  found  by  B\'{a}lint Vir\'{a}g and the first author in unpublished discussions.   Some of the other  combinatorial identities for cumulants that we use are more well-known,  but we give a quick derivation of many of them for self-containment.  

First we recall the notion of cumulants of real-valued random variables (for more details on cumulants and in particular their use in normal approximation, see~\cite{Doring22}). The multilinear moments of a random vector $Y=(Y_1,\ldots ,Y_n)$ are defined by
 \begin{align*}
m_n[Y_1,\ldots ,Y_n] = \E[Y_1\ldots Y_n],  \qquad 
 m_{\pi}[Y_1,\ldots ,Y_n] =\prod_{r=1}^{\ell(\pi)}m_{|\pi(r)|}[Y_{\pi(r)}] \mbox{ for }\pi\in \Pi(m).
 \end{align*}
 The (multilinear) cumulants are defined by
 \begin{align}\label{eq:momentcumulantforrvs}
 \bkappa_n[Y_1,\ldots ,Y_n] =\sum_{\pi\in \Pi(m)} (-1)^{\ell(\pi)-1}(\ell(\pi)-1)!m_{\pi}[Y_1,\ldots ,Y_n], \qquad 
 \bkappa_{\pi}[Y] =\prod_{r=1}^{\ell(\pi)}\bkappa_{|\pi(r)|}[Y_{\pi(r)}].
\end{align}
We also write $\bkappa_n(Y_1) = \bkappa_n[Y_1,\ldots,Y_1]$ for the higher cumulants of a single random variable. \dy{Further, we use the ``inner-product" notation $\langle f,g\rangle$ even when $f \in L^{\infty}$ and $g \in L^1$.}

\mk{Next we state the cumulant reduction result. The motivation for expecting such a result comes from direct computations for the second and third cumulants of the linear statistic $\cX(f)$. The second cumulant was derived in \eqref{eq:reductionformulaformequal2}, and by a similar but more tedious computation one can show that the third cumulant is given by 
\begin{align*}
  \frac{1}{12} \iiint (h(x)+h(y)-2h(z))(h(y)+h(z)-2h(x))(h(z)+h(x)-2h(y))\rho_3^T(x,y,z) \ \md x   \md y  \md z
\end{align*}
The point here is that typically the truncated correlation funtion decays off the diagonal (``asymptotic independence") while the part of the integrand depending on the test function vanishes on the diagonal. Overall this leads to efficient cancellation.}

\begin{proposition}(Cumulant reduction, \cite{Krishnapur2007})
\label{l:cumuredu}
Assume that $\cX$ is a simple point process on $\R^d$ satisfying Assumption \ref{asmpt:integrationofrhokT}.  Let $Y_i=\cX(h_i), i =1,\ldots,m$ for some compactly supported \dy{continuous} functions $h_i:\R^d\to \R,  $ and $Y = (Y_1,\ldots,Y_m)$ be the corresponding random vector.   Then we have that,  
\begin{equation}
\label{eq:cumulant_symmfns}
\bkappa_m[Y] = \<Q_m[(h_i(x_j))_{1\le i,j\le m}], \rho_m^T\>,  
\end{equation}
where 
\[
Q_m[(y_{i,j})_{1\le i,j\le m}]=\sum_{\tau:[m]\to [m-1]}c_{\tau} \prod_{i=1}^{m}(y_{i,\tau(i)}-y_{i,m})
\]
with $c_{\tau}=(-1)^{m-\ell(\tau)}\frac{(\ell(\tau)-1)! \ (m-\ell(\tau))!}{m! \ (m-1)!}$.
\end{proposition}
The key point of this proposition is that $Q_m$ is a function of the differences $y_{i,j}-y_{i,m}$. That is what gives the cancellation that we need. In contrast, here is a more familiar form often used. 
\begin{align}\label{eq:kappaqmpm}
\bkappa_m[Y] &= \sum_{\pi\in \Pi(m)}\<\bigotimes_{q=1}^{\ell(\pi)}\prod_{i\in \pi(q)}h_i \ , \  \rho_{\ell(\pi)}^T\> 
\end{align}
For $h_1=\ldots =h_m$ this can be found in \cite[Lemma~1]{soshnikovclt} for determinantal processes and in \cite[Claim~4.3]{Nazarov12} for general point processes. We give a short derivation of the identity \eqref{eq:kappaqmpm} in Appendix \ref{s:app_cumulant} \mk{as it will be the starting point of the derivation of \eqref{eq:cumulant_symmfns}. The vital role of Assumption~\ref{asmpt:integrationofrhokT} is that each  $\ell$-fold integral w.r.t. $\rho_{\ell}^T$ in \eqref{eq:kappaqmpm} can be extended to an $m$-fold integral w.r.t. $\rho_m^T$, by \eqref{eq:integralsoftruncatedcorrelations}. That allows us to express $\kappa_m$ as an integral of a polynomial of $Y_i$s against $\rho_m^T$. This polynomial is then shown to be a function of the differences alone. }

\begin{proof}[Proof of Proposition \ref{l:cumuredu}]  We start with \eqref{eq:kappaqmpm} and symmetrize over variables to rewrite $\bkappa_m[Y]$ in a more convenient manner.  Observe that if $\tau:[m]\to [m]$ is any function  that  induces the partition $\pi$ (by which we mean that $\tau(i)=\tau(j)$ if and only if $i,j$ belong to the same block of $\pi$), then writing $\ell$ for $\ell(\pi)$ and $\mbox{Ran}(\tau)$  for the range of $\tau$, 
\begin{align*}
\<\bigotimes_{q=1}^{\ell}\prod_{i\in \pi(q)}h_i \ , \  \rho_{\ell}^T\> &=\int_{(\R^d)^{\ell}}\prod_{k=1}^m h_i(x_{\tau(i)}) \ \rho_{\ell}^T(x_{\mbox{\tiny Ran}(\tau)}) \; \md x_{\mbox{\tiny Ran}(\tau)} \\
&=(-1)^{m-\ell}\frac{(\ell-1)!}{(m-1)!}\int_{(\R^d)^{m}}\prod_{i=1}^m h_i(x_{\tau(i)}) \ \rho_{m}^T(x_1,\ldots ,x_m)\md x_1\ldots \md x_m,
\end{align*}
where the last equality follows because of Assumption~\ref{asmpt:integrationofrhokT} (see \eqref{eq:integralsoftruncatedcorrelations}).
 Further, there are $m(m-1)\ldots (m-\ell+1)$ different mappings $\tau$ that give the same partition $\pi$. Consequently, we get (write $\ell(\tau)=\#\mbox{Ran}(\tau)$)
\begin{align}
\bkappa_m[Y] &=\frac{1}{m!}\sum_{\tau:[m]\to [m]}\frac{(-1)^{m-\ell(\tau)}}{\binom{m-1}{\ell(\tau)-1}}\int_{(\R^d)^m}\prod_{i=1}^m h_i(x_{\tau(i)}) \ \rho_{m}^T(x_1,\ldots ,x_m)\md x_1\ldots \md x_m \nonumber \\
&=\<Q_m[(h_i(x_j))_{1\le i,j\le m}], \rho_m^T\>  \label{eq:kappaqmpm1}
\end{align}
where 
\begin{align}
\label{eq:Qmidentity}
Q_m[(\zeta_{i,j})_{1\le i,j\le m}]&=\frac{1}{m!}\sum_{\tau:[m]\to [m]}\frac{(-1)^{m-\ell(\tau)}}{\binom{m-1}{\ell(\tau)-1}}\prod_{i=1}^m\zeta_{i,\tau(i)} 
\end{align}
If we now show that $Q_m$ can be re-written as in the statement of the Proposition, the proof is complete.

Towards the same, we first show that $Q_m[(\zeta_{i,j}+t_i)_{1\le i,j\le m}]$ does not depend on $t_1,\ldots ,t_m$.  By symmetry, it suffices to show the lack of dependence on $t_1$. To this end, set $t_1=t$ and $t_j=0$ for $j\ge 2$ and consider $Q_m[(\zeta_{i,j}+t\delta_{i,1})_{1\le i,j\le m}]$ which is a linear polynomial in $t$. The coefficient of $t$ in this polynomial is
\begin{align*}
 \frac{1}{m!(m-1)!}\sum_{\tau:[m]\to [m]}(-1)^{m-\ell(\tau)}(\ell(\tau)-1)!(m-\ell(\tau))!\prod_{i=2}^n\zeta_{i,\tau(i)}.
\end{align*}
Fix $\tau(2),\ldots ,\tau(m)$, and suppose that $\ell$ of them are distinct. We compute the coefficient of $\prod_{i=2}^n\zeta_{i,\tau(i)}$ in the above expression. 
\begin{enumerate}
\item In the $\ell$ cases when $\tau(1)=\tau(j)$ for some $j\ge 2$, $\ell(\tau)=\ell$. The total contribution to the coefficient from these summands is $(-1)^{m-\ell}(\ell-1)!(m-\ell)!\times \ell$.
\item In the $m-\ell$ cases when $\tau(1)\not=\tau(j)$ for any $j\ge 2$, we have $\ell(\tau)=\ell+1$. The total contribution to the coefficient from these summands is $(-1)^{m-\ell-1}\ell!(m-\ell-1)!\times (m-\ell)$.
\end{enumerate}
Adding the two, we get zero and hence the derivative of $Q_m[(\zeta_{i,j}+t\delta_{i,1})_{1\le i,j\le m}]$ with respect to $t$ is identically zero. This shows that $Q_m[(\zeta_{i,j}+t\delta_{i,1})_{1\le i,j\le m}]$ does not depend on $t$ and $Q_m[(\zeta_{i,j}+t_i)_{1\le i,j\le m}]$ does not depend on $t_1,\ldots ,t_m$ as claimed.

Taking $t_i=-\zeta_{i,m}$, we see that $Q_m[(\zeta_{i,j})_{1\le i,j\le m}] =Q_m[(\zeta_{i,j}-\zeta_{i,m})_{1\le i,j\le m}]$, and using the expression~\eqref{eq:Qmidentity} for the latter gives the expression in the statement of the proposition. Also observe that if $\tau(i) = m$ for some $i \in [m]$, then the corresponding summand vanishes and hence we have restricted to only $\tau : [m] \to [m-1]$.
\end{proof}

\subsection{Proofs of central limit theorems}
\label{s:proofs_clt} 
We prove the two central limit theorems here. In structure, the proof of the first theorem is similar to that of Proposition~\ref{l:covasymstatfindiffsobolev}: We first prove for compactly supported smooth functions and then extend it to functions with less smoothness.

\begin{proof}[Proof of Theorem \ref{thm:cltsmooth}]

\noindent \textsc{\underline{STEP 1:} Proof for compactly supported smooth $f$. \\}
We shall first derive bounds for multilinear cumulants for any compactly supported smooth functions $h_i:\R^d\to \R,  i = 1, \ldots,m$ and then reduce it in the case of $h_i = f,  i = 1,\ldots,m$.  

From Proposition \ref{l:cumuredu}, $\bkappa_m[\cX(h_1),\ldots ,\cX(h_m)]$ is a linear combination of finitely many terms of the form 
$$ \int \prod_{i=1}^{m}(h_i(x_{\tau(i)})-h_i(x_m)) \ \rho_m^T(x_1,\ldots ,x_m) \ \md x_1\ldots \md x_m,$$
for some $\tau:[m]\to [m-1]$. For simplicity, let us first assume that $h_i$ are real analytic and later we will argue for compactly supported smooth functions. 

Then for all $\alpha^{(1)},\ldots,\alpha^{(m)} \in \Z_+^d$, $\int \prod_{i=1}^{m-1}|z_i|^{\alpha^{(i)}}|\kappa_{m-1}|(\md z) < \infty$ and so,
\begin{align*}
&\int \prod_{i=1}^{m}(h_i(x_{\tau(i)})-h_i(x_m)) \ \rho_m^T(x_1,\ldots ,x_m) \ \md x_1\ldots \md x_m \\
&= \int \prod_{i=1}^{m}(h_i(x+z_{\tau(i)})-h_i(x)) \ \kappa_{m-1}(z_1,\ldots ,z_{m-1}) \ \md z_1\ldots \md z_{m-1} \ \md x \\
&= \int \prod_{i=1}^{m}\left(\sum_{\alpha\not=0}\frac{1}{\alpha!}\partial^{\alpha}h_i(x) z_{\tau(i)}^{\alpha}\right) \ \kappa_{m-1}(z_1,\ldots ,z_{m-1}) \ \md z_1\ldots \md z_{m-1}  \ \md x \\
&=\sum_{\alpha^{(1)},\ldots ,\alpha^{(m)}\not=0} I_m(\tau ; \alpha^{(1)},\ldots ,\alpha^{(m)})\int_{\R^d} \prod_{i=1}^m\partial^{\alpha^{(i)}}h_i(x) \; \md x.
\end{align*}
Applying this to $h_{i,L}(x)=h_i(x/L)$ to get
\begin{align*}
&\bkappa_m[\cX_L(h_{1}),\ldots ,\cX_L(h_{m})] =\bkappa_m[\cX(h_{1,L}),\ldots ,\cX(h_{m,L})] 
\\ &=\sum_{\tau:[m]\to [m-1]}c_{\tau} \sum_{\alpha^{(1)},\ldots ,\alpha^{(m)}\not=0} I_m(\tau ; \alpha^{(1)},\ldots ,\alpha^{(m)})L^{d-|\alpha^{(1)}|-\ldots -|\alpha^{(m)}|}\int_{\R^d} \prod_{i=1}^m\partial^{\alpha^{(i)}}h_i.
\end{align*}
By definition of $d_m$,  we see that $\bkappa_m[h_{1,L},\ldots ,h_{m,L}]=O(L^{d-d_m})$ and putting $h_i = f$ for all $i$, we obtain $\bkappa_m(\cX(f_L)) = O(L^{d-d_m})$. Thus 
\begin{align}\label{eq:cumulantbdfromsecondone}
\bkappa_m\Big( \frac{\cX(f_L) - \E[\cX(f_L)]}{\sqrt{L^{d-2k}}} \Big) &= \frac{\bkappa_m[\cX(f_L)]}{L^{m(d-2k)/2}} \\
&= O(L^{(d-d_m)-\frac12 m(d-2k)}).
\end{align}
By our assumption that $2(d-d_m)<m(d-2k)$  for all large  $m$, it follows that all large cumulants vanish in the limit. \dy{Further, $\bkappa_2\Big( \frac{\cX(f_L) - \E[\cX(f_L)]}{\sqrt{L^{d-2k}}} \Big) \to Q^{\kappa}_{2k}(f,f)$ as $L \to \infty$ by Proposition~\ref{l:covasymstatfindiffsobolev}(2). Now using a variant of Marcinkiewicz's theorem (see \cite[Theorem 1]{janson1988}), we can conclude that the limiting distribution is $N(0,Q^{\kappa}_{2k}(f,f))$.}


Now we fix the proof for compactly supported smooth functions $h_i$ and more importantly only assuming integrability of $|z|^{\gamma}\kappa_{m-1}(z)$ for $|\gamma|\le l+1$ with $l =d_m-1$. In this case, we fix $\tau : [m] \to [m-1]$ as before and define $H_x(z)= \prod_{i=1}^{m}(h_i(x+z_{\tau(i)})-h_i(x))$. The Taylor expansion of $H_x:(\R^d)^{m-1}\to \R$ around $0$ is 
\begin{align}\label{eq:tayloexpansionHx}
H_x(z_1,\ldots ,z_{m-1}) = \sum_{1\le |\theta|\le l}\frac{(-1)^{|\theta|}}{\theta!}\partial^{\theta}H_x(0) \prod_{j=1}^{m-1}z_{j}^{\theta^{(j)}} \ + \ \sum_{|\theta|=l+1}\frac{(-1)^{|\theta|}|\theta|}{\theta!}R_{\theta}(x,z)\prod_{j=1}^{m-1}z_{j}^{\theta^{(j)}}.
\end{align}
Here $\theta\in (\Z_+^d)^{m-1}$ is written as $\theta=(\theta^{(1)},\ldots ,\theta^{(m-1)})$ with $\theta^{(j)}\in \Z_+^d$, and the self-explanatory meaning $|\theta|=|\theta^{(1)}|+\ldots +|\theta^{(m-1)}|$, $\theta!=\theta^{(1)}!\ldots \theta^{(m-1)}!$ and $\partial^{\theta}=\partial_{z_1}^{\theta^{(1)}}\ldots \partial_{z_{m-1}}^{\theta^{(m-1)}}$. The smoothness of $h$ gives that $|R_{\theta}(x,z)| \leq C_{\theta},$ a finite constant.

Thus, for a given $\theta\in (\Z_+^d)^{m-1}$, it is easy to see that $\partial^{\theta}H_x(0)$ is a linear combination of terms of the form $\prod_{i=1}^m\partial^{\alpha^{(i)}}h_i(x)$, where $\alpha\in (\Z_+^d)^m$ satisfies $|\alpha|=|\theta|$. Therefore, the integral of the first summand in \eqref{eq:tayloexpansionHx} with respect to $\kappa_{m-1}(z_1,\ldots,z_{m-1}) \, \md z_1 \ldots \md z_{m-1}$ is a sum of terms of the form 
\[
I_m(\tau ; \alpha^{(1)},\ldots ,\alpha^{(m)}) \int \prod_{i=1}^m\partial^{\alpha^{(i)}}h_i(x) \md x
\]
over $\alpha\in (\Z_+^d)^m$ satisfying $|\alpha|\le l$.
Since $l = d_m-1$, for all $|\alpha| < l$, $I_m(\tau ; \alpha^{(1)},\ldots ,\alpha^{(m)}) = 0$  and so we are only left with the second term in \eqref{eq:tayloexpansionHx}. Now put $h_{i,L}$ in \eqref{eq:tayloexpansionHx} and correspondingly $H_{L,x}$ to obtain by a change of variables that
\begin{align*}
& |\int H_{L,x}(z_1,\ldots ,z_{m-1}) \kappa_{m-1}(z_1,\ldots,z_{m-1}) \md z_1 \ldots \md z_{m-1} \md x | \\
& \leq  \sum_{|\theta|=d_m} L^{d-|\theta^{(1)}|-\ldots-|\theta^{(m)}|} \frac{(-1)^{|\theta|}|\theta|}{\theta!} \int |R_{\theta}(x,\frac{z}{L})|\prod_{j=1}^{m}|z_{\tau(j)}|^{\theta^{(j)}} |\kappa_{m-1}(z_1,\ldots,z_{m-1})| \md z_1 \ldots \md z_{m-1} \md x \\
& = O(L^{d-d_m}),
\end{align*}
where for the last bound we have used the boundedness of $R_{\theta}$ and the integrability assumption on $\kappa_{m-1}$. Thus, we derive as in \eqref{eq:cumulantbdfromsecondone} that 
$$\bkappa_m\Big( \frac{\cX(f_L) - \E[\cX(f_L)]}{\sqrt{L^{d-2k}}} \Big)=O(L^{(d-d_m)-\frac12 m(d-2k)}),$$
and thus the required central limit theorem follows.

\medskip

\noindent \textsc{\underline{STEP 2:} Proof for $f \in L^1 \cap H^k$.} \\
As in the second step of the proof of Proposition~\ref{l:covasymstatfindiffsobolev}, we approximate $f\in L^1\cap H^k$ by compactly supported smooth functions. More precisely, for any $\varepsilon>0$, we can write  $f=g+h$ with $g\in C_c^{\infty}(\R^d)$ and $Q_{2k}(h,h)<\varepsilon$. Then we also have that $|Q_{2k}(f,f)-Q_{2k}(g,g)|\le C\sqrt{\varepsilon}$, by triangle and Cauchy-Schwarz inequality. For any $L>0$, 
\begin{align*}
\frac{\cX(f_L) - \E[\cX(f_L)]}{\sqrt{L^{d-2k}}}=\frac{\cX(g_L) - \E[\cX(g_L)]}{\sqrt{L^{d-2k}}}+\frac{\cX(h_L)-\E[\cX(h_L)]}{\sqrt{L^{d-2k}}}.
\end{align*}
The second summand has zero mean and its variance converges to $Q_{2k}(h,h)<\varepsilon$. The first summand converges in distribution to $N(0,Q_{2k}(g,g))$. This shows that all subsequential limits of  $\frac{\cX(f_L) - \E[\cX(f_L)]}{\sqrt{L^{d-2k}}}$ are within $O(\sqrt{\eps})$ of $N(0,Q_{2k}(f,f))$ in the Kantorovich distance $W_2$. As $\varepsilon \to 0$, it follows that $\frac{\cX(f_L) - \E[\cX(f_L)]}{\sqrt{L^{d-2k}}}$ converges in distribution to $N(0,Q_{2k}(f,f))$.
\end{proof}

Next we prove the CLT for H\"{o}lder statistics.
\begin{proof}[Proof of Theorem \ref{thm:cltholder}]
From Proposition \ref{l:cumuredu} and stationarity, $\bkappa_m(\cX_L(f))$ is a sum of finitely many terms of the kind (for some $\tau : [m] \to [m-1]$)
\begin{align}
& \Big| \int \prod_{i=1}^{m}(f_{L}(x+z_{\tau(i)})-f_{L}(x)) \ \kappa_{m-1}(z_1,\ldots ,z_{m-1}) \ \md z_1\ldots \md z_{m-1} \ \md x \Big| \nonumber \\
\label{e:kminthol} & \le C^m L^{-\alpha m} \int_{L.B}\int_{(L.S)^{m-1}}\prod_{i=1}^{m-1}|z_i|^{\alpha\ell_i} \ |\kappa_{m-1}|(z_1,\ldots ,z_{m-1}) \ \md x
\end{align}
where $S=B-B$  where $B$ is the support of $f$ and $\ell_i=\#\{k \; : \; \tau(k)=i\}$ and $C$ is the $\alpha$-H\"{o}lder norm of $f$. Since $\sum l_i = m$, the integrability assumption \eqref{eq:intgrconditionforCLTHol} grants finiteness of the above integrals in \eqref{e:kminthol} and so we bound the inner integral in \eqref{e:kminthol} by
\[
\int \prod_{i=1}^{m-1}|z_i|^{\alpha\ell_i} \ |\kappa_{m-1}|(z_1,\ldots ,z_{m-1}) \ \md z_1 \ldots \md z_{m-1}. 
\]
Denoting the supremum over all choices of $l_i$ in the above integral by $C_m$, we can then bound
\eqref{e:kminthol} by
\[
C_m\times C^{m}L^{-\alpha m}\times L^d|B| = C_{m,f}'L^{-\alpha m+d}
\]
where $C_{m,f}'$ depends on $f$ and $m$, but not on $L$. As we have assumed that $\bkappa_2(\cX_L(f)) \geq CL^{-2\alpha + \epsilon}$, it follows that for large $m$, $\bkappa_m(\cX_L(f)) = o(\bkappa_2(\cX_L(f))^{m/2})$  as $L\to \infty$. \dy{Since the second cumulant is the variance, we have that $\bkappa_m(\frac{\cX_L(f)) - \EXP{\cX_L(f)}}{\sqrt{\Var{\cX_L(f)}}}) \to 0$ for large enough $m$. Thus again by the variant of Marcinkiewicz's theorem (see \cite[Theorem 1]{janson1988}), we can conclude the normal convergence.}
\end{proof}

\appendix
\section{Appendix: Proof of the cumulant formula~\eqref{eq:kappaqmpm}}
\label{s:app_cumulant}
Let $h_1,\ldots,h_m$ be compactly supported continuous functions. Using the formula for multilinear moments and the Campbell-Mecke formula, we have
\begin{align*}
\E [Y_1 \ldots Y_m] & = \E \left[\sum_{x_1,\ldots ,x_m\in \cX}h_1(x_1)\ldots h_m(x_m) \right] \; = \; \sum_{\pi\in \Pi(m)} \langle \bigotimes_{r=1}^{\ell(\pi)}\prod_{i\in \pi(r)}h_i \ , \ \rho_{\ell(\pi)} \rangle
\end{align*}
by summing  according to the induced partition of $[m]$ in which $i$ and $j$ are in the same block if and only if $x_i=x_j$. To be clear, the term corresponding to $\pi$ with $\ell(\pi)=\ell$ is
\begin{align*}
    \int_{(\R^d)^{\ell}} \rho_{\ell}(y_1,\ldots ,y_{\ell})   \prod_{r=1}^{\ell}\prod_{i\in \pi(r)}h_i(y_r)  \ \md y_1\ldots \md y_{\ell}.
\end{align*}
\dy{Here we have used the assumption that $\rho_{\ell}$ are absolutely integrable, which ensures that the above integrals exist and are finite. This allows the use of the Campbell-Mecke formula as well.}

For $\pi,\sigma\in \Pi(m)$, we write $\pi\leq \sigma$ if  $\pi$ refines $\sigma$, i.e., each block of $\pi$ is contained in a block of $\sigma$. In such a case, the restriction $\pi\vert_{\sigma(r)}$ is a set partition of $\sigma(r)$, for each $r\le \ell(\sigma)$.
\begin{align*}
\bkappa_m[Y] &= \sum_{\sigma\in \Pi(m)}(-1)^{\ell(\sigma)-1}(\ell(\sigma)-1)!\prod_{r=1}^{\ell(\sigma)}\sum_{\pi_r\in \Pi(\sigma_r)}\<\bigotimes_{s=1}^{\ell(\pi_r)}\prod_{i\in \pi_r(s)}h_i \ , \ \rho_{\ell(\pi_r)}\> \\
&= \sum_{\sigma\in \Pi(m)}(-1)^{\ell(\sigma)-1}(\ell(\sigma)-1)!\sum_{\pi\in \Pi(m): \pi\le \sigma}\prod_{r=1}^{\ell(\sigma)}\<\bigotimes_{s=1}^{\ell(\pi\vert_{\sigma(r)})}\prod_{i\in \pi\vert_{\sigma(r)}}h_i \ , \ \rho_{\ell(\pi\vert_{\sigma(r)})}\> 
\end{align*}
\begin{align*}
&= \sum_{\sigma\in \Pi(m)}(-1)^{\ell(\sigma)-1}(\ell(\sigma)-1)!\sum_{\pi\in \Pi(m): \pi\le \sigma}\<\bigotimes_{q=1}^{\ell(\pi)}\prod_{i\in \pi(q)}h_i \ , \ \bigotimes_{r=1}^{\ell(\sigma)}\rho_{\ell(\pi\vert_{\sigma(r)})}\> \\
&= \sum_{\pi\in \Pi(m)}\<\bigotimes_{q=1}^{\ell(\pi)}\prod_{i\in \pi(q)}h_i \ , \  \sum_{\sigma\in \Pi(m):\sigma\ge \pi}(-1)^{\ell(\sigma)-1}(\ell(\sigma)-1)! \bigotimes_{r=1}^{\ell(\sigma)}\rho_{\ell(\pi\vert_{\sigma(r)})} \> \\
&= \sum_{\pi\in \Pi(m)}\<\bigotimes_{q=1}^{\ell(\pi)}\prod_{i\in \pi(q)}h_i \ , \  \rho_{\ell(\pi)}^T\>.
\end{align*}
When $h_i=f$  for all $i$,  by the symmetry of $\rho_{\ell}^T$,  this expression specializes to 
\begin{align}\label{eq:cumulkantoflinstat}
\bkappa_m[Y] &= \sum_{\pi\in \Pi(m)} \<f^{|\pi_1|}\otimes \ldots \otimes f^{|\pi_{\ell}|},\rho_{\ell}^T\>
\end{align}
in which form it occurs in the  literature, for example in Claim~4.3 of \cite{Nazarov12}.

\section{Appendix: Simple point processes on discrete spaces}
\label{s:app_discrete}
Let $E$ be a countable set and let $\cX$ be a simple point process with a fixed $N$ number of points, i.e., $\cX(E)=N$.  All information about the point process is also encoded in the collection of Bernoulli random variables $(V_x)_{x\in E}$, where  $V_x=\cX(\{x\})$. The correlation functions are defined as 
\[
\rho_k(x_1,\ldots ,x_k)=\begin{cases} \E[V_{x_1}\ldots V_{x_k}] &\mbox{ if }x_i\mbox{ are distinct,} \\ 0 & \mbox{ if not}.
\end{cases}
\]
We use this in \eqref{eq:defnoftruncatedcorrelations} to define $\rho_k^T(x_1,\ldots ,x_k)$ for any $x_1,\ldots ,x_k\in E$, not necessarily distinct. When $x_1,\ldots ,x_k$ are distinct, from the definition of $\rho_k$ and relationship between moments and cumulants~\eqref{eq:momentcumulantforrvs}, it immediately follows that
\begin{align}\label{eq:rhokTfordistinct}
\rho_k^T(x_1,\ldots ,x_k)=\bkappa_k[V_{x_1},\ldots ,V_{x_k}].
\end{align}
What happens when $x_1,\ldots ,x_k$ are not distinct? The general case is outlined in Remark~\ref{rem:cumulantgeneralform} below. Here we restrict ourselves to the special case where  $x_1,\ldots ,x_{k-1}$ are distinct and $x_k=x_j$ for some $j<k$.  We claim that 
\begin{align}\label{eq:rhokTfornondistinct}
\rho_k^T(x_1,\ldots ,x_{k-1},x_{j})=\bkappa_k[V_{x_1},\ldots ,V_{x_{k-1}},V_{x_j}]-\bkappa_{k-1}[V_{x_1},\ldots ,V_{x_{k-1}}]. 
\end{align}
To see this, take $j=1$ without loss of generality. Observe that in \eqref{eq:defnoftruncatedcorrelations}, the summands corresponding to $\pi$ where $1$ and $k$ belong to the same block, vanish. Therefore,
\begin{align*}
\rho_k^T(x_1,\ldots ,x_k) &= \bkappa_k[V_{x_1},\ldots ,V_{x_k}] - \sum_{\pi\in \Pi(k)}^{*}(-1)^{\ell(\pi)-1}(\ell(\pi)-1)!\prod_{j=1}^{\ell(\pi)}\E\left[\prod_{i\in \pi_j}V_{x_i}\right],
\end{align*}
where the second sum is over $\pi$ in which $1$ and $k$ belong to the same block. As $V_{x_1}V_{x_k}=V_{x_1}$, the second sum remains the same if we drop $V_{x_k}$, which shows that it is just the same as $\bkappa_{k-1}[V_{x_1},\ldots ,V_{x_{k-1}}]$. This proves \eqref{eq:rhokTfornondistinct}. 

\paragraph{\dy{Verification of Assumption~\ref{asmpt:integrationofrhokT}:}} For distinct $x_1,\ldots ,x_{k-1}$, using \eqref{eq:rhokTfordistinct} and \eqref{eq:rhokTfornondistinct}, we get
\begin{align*}
\sum_{u}\rho_k^T(x_1,\ldots ,x_{k-1},u) &= \sum_{u}\bkappa_k[V_{x_1},\ldots ,V_{x_{k-1}},V_{u}] - \sum_{j=1}^{k-1}\bkappa_{k-1}[V_{x_1},\ldots ,V_{x_{k-1}}] \\
&=\bkappa_k[V_{x_1},\ldots ,V_{x_{k-1}},\sum_{u}V_{u}] - (k-1)\bkappa_{k-1}[V_{x_1},\ldots ,V_{x_{k-1}}]
\end{align*}
using the multilinearity of $\bkappa_k$. As $\sum_{u}V_u=N$, a constant, the first term vanishes. Using \eqref{eq:rhokTfordistinct} again, the second term is just $-(k-1)\rho_{k-1}^T(x_1,\ldots ,x_{k-1})$, which proves Assumption~\ref{asmpt:integrationofrhokT}. 

\begin{remark}\label{rem:cumulantgeneralform} Using the inclusion-exclusion principle, it may be checked that for any $x_1,\ldots ,x_k$,
\begin{align*}
    \rho_k^{T}(x_1,\ldots ,x_k)=\sum_{1\le j_1<\ldots <j_r\le k}^*(-1)^{k-r}\kappa_r[V_{x_{j_1}},\ldots ,V_{x_{j_r}}]
\end{align*}
where the sum is over all $1\le j_1<\ldots <j_r\le k$ such that $\{x_{j_1},\ldots ,x_{j_r}\}=\{x_1,\ldots ,x_k\}$ (i.e., all distinct $x_i$s must occur among $x_{j_1},\ldots ,x_{j_r}$).
\end{remark}

\section{Appendix: On truncated correlations for zeros of random analytic functions}\label{app_GAFcumulants}
Let $F$ be a random analytic function on a domain $\Omega\subseteq \mathbb C$ and let $\cX$ be the point process of its zero set. Some of the manipulations below (including that $\cX$ is a simple point process a.s.) require some mild conditions on $F$, but they are all satisfied for Gaussian analytic functions that have no deterministic zeros and sufficient non-degeneracy (in the sense of Nazarov and Sodin~\cite{Nazarov12}). From Green's theorem, $\cX=\frac{1}{2\pi}\Delta_z\log|F(z)|$ in the distributional sense, and this is the source of the formulas \eqref{eq:truncorrforGAF} and \eqref{eq:truncorrforGAF2} for the truncated correlations of $\cX$. 

In an attempt to prove \eqref{eq:integralsoftruncatedcorrelations0}, we  rephrase it as follows (compare this with the proof in Appendix~\ref{s:app_discrete} for finite processes on discrete spaces). 

 Fix distinct $z_1,\ldots ,z_{k-1}$ and write $h(z_k)=\Delta_{z_1}\ldots \Delta_{z_{k-1}} \bkappa_k[\log|\widetilde{F}(z_1)|,\ldots ,\log|\widetilde{F}(z_k)|]$. 
 With  $\Omega_{\eps}=\mathbb C\setminus \cup_{i=1}^{k-1}\mathbb D(z_i,\eps)$ and using \eqref{eq:truncorrforGAF2} we derive that
\begin{align*}
(2\pi)^k\int \rho_k^T(z_1,\ldots ,z_{k-1},z_{k})\md z_k &= \lim_{\eps\downarrow 0}\int_{\Omega_{\eps}} \Delta_z h(z) \md z \\
&= \lim_{\eps\downarrow 0}\sum_{j=1}^{k-1}\int_{0}^{2\pi} \frac{\partial}{\partial \eps}h(z_j+\eps e^{i\theta}) \  \eps \md\theta 
\end{align*}
 by Green's theorem (assuming fast decay at infinity or at the boundary of $\Omega$). Thus, using the symmetry of the cumulant $\bkappa_k$ in its arguments, \eqref{eq:integralsoftruncatedcorrelations0} will follow if we show that for each $k$ and distinct $z_1,\ldots ,z_{k-1}$,
\begin{align}
& \lim_{\eps\downarrow 0}\int_{0}^{2\pi}  \eps \, \frac{\partial}{\partial \eps} \Delta_{z_1}\ldots \Delta_{z_{k-1}} \bkappa_k[\log|\widetilde{F}(z_1)|,\ldots ,\log|\widetilde{F}(z_{k-1})|,\log|\widetilde{F}(z_{1}+\eps e^{i\theta})|] \, \frac{\md \theta}{2\pi} \nonumber \\ 
 \label{eq:reducedidentityforRAF} & = -\Delta_{z_1}\ldots \Delta_{z_{k-1}}\bkappa_{k-1}[\log|\widetilde{F}(z_1)|,\ldots ,\log|\widetilde{F}(z_{k-1})|].
\end{align}
The advantage of  \eqref{eq:reducedidentityforRAF} over \eqref{eq:integralsoftruncatedcorrelations0} is that here the contributions of the integrals are localized around the points $z_1,\ldots ,z_{k-1}$. It is also  similar to the derivation in Appendix~\ref{s:app_discrete}. 

Nevertheless, we have not been able to prove \eqref{eq:reducedidentityforRAF} for G.E.F., let alone more general Gaussian analytic functions. In the main text (Example~\ref{ex:GEFzerostrunccorr}) we sketched a proof for $k=2$ for G.E.F. Here  we present a proof for $k=2$ for general Gaussian analytic functions.
\begin{proof}[Proof of \eqref{eq:reducedidentityforRAF} for $k=2$ for Gaussian analytic functions] Let $K(z,w)=\E[F(z)\overline{F(w)}]$ be the covariance and let $\theta(z,w)=\E[\widetilde{F}(z)\overline{\widetilde{F}(w)}]=\frac{K(z,w)}{\sqrt{K(z,z)}\sqrt{K(w,w)}}$ be the correlation. The assumption of no deterministic zeros implies $K(z,z) \neq 0$ for all $z$ and hence $\theta$ is well-defined and smooth. 

Now (see Section~3.5 in \cite{Ben09}, in particular, Lemma~3.5.2), we can write
\begin{align*}
    \bkappa_2[\log|\widetilde{F}(z)|,\log|\widetilde{F}(w)|]&=\frac14 \Li(|\theta(z,w)|^2)
\end{align*}
where $\Li(u)=\sum_{p\ge 1}\frac{u^p}{p^2}$ is the dilogarithm function. Observe that $|\theta(z,w)|^2\to 1$ as $w\to z$. Therefore, to write a perturbative expression, we rewrite the well-known reflection identity (as the two sides are equal for $u=0$, it suffices to show that the derivatives are equal, and that follows from  $\Li'(u)=-\frac{1}{u}\log(1-u)$)
\begin{align*}
    \Li(u)+\Li(1-u)=\frac{\pi^2}{6}-(\log u)\log(1-u)
\end{align*}
as 
\begin{align*}
    \Li(u)=\frac{\pi^2}{6}+\sum_{p=1}^{\infty}\frac{(1-u)^p}{p^2}(p\log(1-u)-1)
\end{align*}
in order to write
\begin{align*}
&    \bkappa_2[\log|\widetilde{F}(z)|,\log|\widetilde{F}(w)|] \\
&=\frac{\pi^2}{24}+\frac14(1-|\theta(z,w)|^2)(\log(1-|\theta(z,w)|^2)-1)  + O((1-|\theta(z,w)|^2)^2\log(1-|\theta(z,w)|^2)).
\end{align*}
Write the power series expansion $K(w+\xi,w+\zeta)=\sum_{p,q\ge 0}a_{p,q}\xi^{p}\overline{\zeta}^{q}$. A little work shows that
\begin{align*}
1-|\theta(z,w)|^2 &= -\frac{a_{1,1}a_{0,0}-a_{1,0}a_{0,1}}{a_{0,0}^2}|z-w|^2+ o(|z-w|^2) \mbox{ as }z\to w \\
&= -\pi \rho_1(w)|z-w|^2+\ldots
\end{align*}
We used  the Kac-Rice formula (see page~26 of \cite{Ben09}) in arriving at the last line.  In addition note that $a_{0,0} = K(w,w) > 0$ and hence $\rho_1(w)$ is a smooth function of $w$. Therefore, as $z\to w$,
\begin{align*}
\bkappa_2[\log|\widetilde{F}(z)|,\log|\widetilde{F}(w)|] &= \frac{\pi^2}{24}-\frac{\pi}{4} \rho_1(w)|z-w|^2\log|z-w|^2+[\mbox{lower order terms}].
\end{align*}
A direct calculation shows that
\begin{align*}
    \Delta_w (|z-w|^2\log|z-w|^2)&= 4(2+\log|z-w|^2).
\end{align*}
When we take $z=w+\eps e^{i\theta}$ and differentiate with respect to $\eps$, this becomes $\frac{8}{\eps}$. 

What about the other terms in $\bkappa_2[\log|\widetilde{F}(z)|,\log|\widetilde{F}(w)|]$? The lower order terms are not merely smaller in magnitude, they are terms involving higher powers of $|z-w|^2$. Hence if any of the differentiation operations is  applied to them, the overall contribution is $o(1/\eps)$. The same holds if any of the derivatives is applied to  the $\rho_1(w)$ factor (which is smooth as noted above).  Consequently, 
\begin{align*}
    \frac{\partial}{\partial \eps} \Delta_{w}\bkappa_2[\log|\widetilde{F}(w)|,\log|\widetilde{F}(w+\eps e^{i\theta})|]=-\frac{2\pi}{\eps}\rho_1(w)+o(1/\eps).
\end{align*}
Multiplying by $\eps$ and integrating w.r.t. $d\theta/2\pi$  gives $-2\pi\rho_1(w)$. As $\rho_1(w)=\frac{1}{2\pi} \bkappa_1[\log|\widetilde{F}(w)|]$,  we arrive at  \eqref{eq:reducedidentityforRAF} for $k=2$.
\end{proof}

\section*{Acknowledgements}
DY's research was partially supported by SERB-MATRICS Grant MTR/2020/000470 and CPDA from the Indian Statistical Institute. M.K.'s research was partly supported by the DST-FIST program 2021[TPN-70661]. The authors are also thankful to M. A. Klatt and R. Lachi\`eze-Rey for comments on a preliminary version. The authors thank an anonymous referee and Apoorva Khare whose comments have led to a much improved presentation.

\end{document}